\documentclass[10pt]{amsart}
\usepackage{pgfplots}
\usepackage{tikz}

\usetikzlibrary{automata}
\usetikzlibrary{positioning}
\usepackage{amsmath, amssymb} 
\usepackage{amsthm}
\usepackage{graphicx}
\usepackage[shortlabels]{enumitem}

\usepackage{amsfonts,geometry,hyperref}

\newtheorem{conj}{Conjecture}[section]

\newtheorem{remark}{Remark}[section]

\newtheorem{theorem}{Theorem}[section]

\newtheorem{lemma}[theorem]{Lemma}

\DeclareMathOperator{\Spec}{Spec}
\DeclareMathOperator{\L-Spec}{L-spec}
\DeclareMathOperator{\Q-Spec}{Q-spec}
\DeclareMathOperator{\CN-Spec}{CN-spec}
\DeclareMathOperator{\con}{con}

\DeclareMathOperator{\Nbd}{Nbd}
\DeclareMathOperator{\ABC}{ABC}
\DeclareMathOperator{\GA}{GA}
\DeclareMathOperator{\SCI}{SCI}

\allowdisplaybreaks
\begin{document}
\title[Topological indices and spectral properties of SGB-graphs]{Certain topological indices and spectral properties of SGB-graphs of finite cyclic groups}

\author[S. Das, A. Erfanian and R. K. Nath]{Shrabani Das, Ahmad Erfanian and Rajat Kanti Nath$^*$}

\address{S. Das, Department of Mathematical Sciences, Tezpur University, Napaam-784028, Sonitpur, Assam, India.
\newline
Department of Mathematics, Sibsagar University, Joysagar-785665, Sibsagar, Assam, India.}

\email{shrabanidas904@gmail.com}

\address{A. Erfanian, Department of Pure Mathematics, Faculty of Mathematical Sciences, Ferdowsi University of Mashhad,
	P.O. Box 1159-91775,
	Mashhad, Iran}
	\email{erfanian@um.ac.ir}

\address{R. K. Nath, Department of Mathematical Sciences, Tezpur University, Napaam-784028, Sonitpur, Assam, India.} 
\email{ rajatkantinath@yahoo.com}
\thanks{$^*$Corresponding Author}
\begin{abstract}
Let $L(G)$ be the set of all subgroups of a group $G$. The subgroup generating bipartite graph $\mathcal{B}(G)$ defined on $G$ is a bipartite graph whose vertex set is the union of two sets $G \times G$ and $L(G)$, and two vertices $(a, b) \in G \times G$ and $H \in L(G)$ are adjacent if $H$ is generated by $a$ and $b$. In this paper, we realize the structures of $\mathcal{B}(G)$ for cyclic groups of order $pq, p^2q$ and $p^2q^2$, where $p$ and $q$ are primes and $p \neq q$. We also deduce expressions for first and second Zagreb indices of these graphs and check the validity of Hansen-Vuki{\v{c}}evi{\'c} conjecture [Hansen, P. and Vuki{\v{c}}evi{\'c}, D. Comparing the Zagreb indices, {\em Croatica Chemica Acta}, \textbf{80}(2), 165-168, 2007]. Expressions of certain other degree-based topological indices of these graphs are also computed. We further compute various spectra and their corresponding energies of $\mathcal{B}(G)$ if $G$ is any cyclic group of order $p^n, pq, p^2q$ and $p^2q^2$, where $p$ and $q$ are two distinct primes and $n \geq 1$. We conclude the paper showing that $\mathcal{B}(G)$ satisfies E-LE conjecture [Gutman, I., Abreu, N. M. M., Vinagre, C. T. M., Bonifacioa, A. S. and Radenkovic, S. Relation between energy and Laplacian energy, {\em MATCH Communications in Mathematical and in Computer Chemistry}, \textbf{59}, 343--354, 2008] for these groups.
\end{abstract}

\thanks{ }

\maketitle

\section{Introduction}
Defining graphs on finite groups and studying them towards the characterization of finite groups/graphs has been an active area of research in recent times. Cayley graph, generating graph, power graph, commuting graph, B superA graph etc. (see \cite{cameron2021graphs, ANC-2022}) are  some examples of such graphs where the vertex set of the graph consists of the elements of the group. Other classes of graphs defined on groups whose vertices are the orders of the elements or the orders of the conjugacy classes have also been considered (see \cite{MN-2024, Lewis-2008}). There are graphs   defined on groups by considering the vertex set as the set of subgroups of the group. Intersection graph \cite{CP69}, inclusion graph \cite{DR16} and permutability graph \cite{RD14} are some examples of such graphs. Recently, a bipartite graph has been defined on a group by using its automorphism group \cite{MM-AE-RM-23}. Let $G$ be a finite group and $L(G) := \{H : H \text{ is a subgroup of } G\}$. In \cite{DEN-23}, we have introduced the subgroup generating bipartite graph (abbreviated as SGB-graph and denoted by $\mathcal{B}(G)$)    whose vertex set $V(\mathcal{B}(G))$ is the union of two sets $G \times G$ and $L(G)$, and two vertices $(a, b) \in G \times G$ and $H \in L(G)$ are adjacent if  $H = \langle a, b \rangle$, the subgroup generated by $a$ and $b$. One of the objectives in considering SGB-graph is its connection with 
% The probability generating  a given subgroup $H$ of $G$, denoted by $\Pr_H(G)$, is the probability that a randomly chosen pair of elements of $G$ generate $H$. The origin of $\Pr_H(G)$ lies in a paper of Hall \cite{Hall36} and a generalized version of which was studied in \cite{Pak99}. We have  obtained relations between $\mathcal{B}(G)$  and 
various probabilities such as probability generating  a given subgroup \cite{Di69}, commuting probability \cite{DNP-13}, cyclicity degree \cite{PSSW93}, nilpotency degree \cite{DGMW92}, solvability degree \cite{FGSV2000} associated to finite groups. 
%For any subgroup $H$ of $G$, it was shown  (see \cite[Lemma 3.1]{DEN-23}) that
%\begin{equation}\label{deg(H in L(G))}
%	\deg_{\mathcal{B}(G)}(H)=|G|^2 {\Pr}_H(G),
%\end{equation}
%where $\deg_{\mathcal{G}}(x)$ is the degree of any vertex $x$ in any graph $\mathcal{G}$.
In \cite{DEN-23}, we have also established connections between $\mathcal{B}(G)$ and generating graph of $G$ and discussed about various graph parameters of $\mathcal{B}(G)$ such as independence number, domination number, girth, diameter etc. Subsequently, we obtained Zagreb indices and various energies of SGB-graphs of certain dihedral and dicyclic groups in \cite{DEN-24} and \cite{DEN-24-3} respectively.
 
In Section 2, we realize the structures of  $\mathcal{B}(G)$ when $G$ is a cyclic group of order $pq, p^2q$ and $p^2q^2$ for any two distinct primes $p$ and $q$. In Section 3, we obtain first Zagreb index and second Zagreb index of $\mathcal{B}(G)$ for these graphs and check the validity of Hansen-Vuki{\v{c}}evi{\'c} conjecture \cite{hansen2007comparing}. While computing Zagreb indices of $\mathcal{B}(G)$ we have computed $\deg_{\mathcal{B}(G)}(H)$ for all $H \in L(G)$ for the above mentioned groups. Using these information, in Section 4, we  also compute Randic Connectivity index, Atom-Bond Connectivity index, Geometric-Arithmetic index, Harmonic index and Sum-Connectivity index of $\mathcal{B}(G)$. In Section 5, we compute spectrum, Laplacian spectrum,  signless Laplacian spectrum, common neighbourhood spectrum and their corresponding energies of $\mathcal{B}(G)$ for the groups considered in Section 2 along with cyclic group of order $p^n$, where $p$ is any prime and $n \geq 1$. We also show that $\mathcal{B}(G)$ is hypoenergetic but neither hyperenergetic, L-hyperenergetic, Q-hyperenergetic nor CN-hyperenergetic if $G$ is one of these groups. We conclude the paper showing that  $\mathcal{B}(G)$ satisfies E-LE conjecture \cite{E-LE-Gutman} for these groups.
 \section{Structures of $\mathcal{B}(G)$}
 In \cite{DEN-23, DEN-24}, Das et al. obtained the structures of $\mathcal{B}(G)$ when $G$ is a cyclic group of order $2p, 2p^2, 4p$ and $4p^2$ where $p$ is an odd prime. They also obtained the structure of $\mathcal{B}(G)$ when $G$ is a cyclic group of order $p^n$, where $p$ is any prime and $n \geq 1$. In this section, we obtain the structures of $\mathcal{B}(G)$ when $G$ is a cyclic group of order $pq, p^2q$ and $p^2q^2$, where $p$ and $q$ are two primes such that $p \neq q$. We write $\sqcup$ to denote disjoint union of sets or graphs and $mK_{1, r}$ to denote $m$ copies of the star $K_{1, r}$. Moreover, we write $\Nbd_G(x)$ to denote the neighborhood of $x \in V(G)$ which is the set of all $y \in V(G)$ such that $y$ is adjacent to $x$. Also, $G[S]$ denotes the subgraph of G induced by a subset $S$ of $V(G)$.
 We begin this section by recalling the following result. 
 \begin{theorem}\label{structure_cyclic_p^n}
\cite{DEN-24} If $G$ is a cyclic group of order $p^n$, where $p$ is a prime and $n \geq 1$, then $\mathcal{B}(G)=K_2 \sqcup K_{1, p^2-1} \sqcup K_{1, p^2(p^2-1)} \sqcup \cdots \sqcup K_{1, p^{2n-2}(p^2-1)}$.
\end{theorem}
We remark that one may choose the primes $p$ and $q$ such that $p<q$ without loss of generality while considering cyclic groups of order $pq$ and $p^2q^2$. 
\begin{theorem}\label{structure_cyclic_pq}
If $G$ is a cyclic group of order $pq$, where $p$, $q$ are two  primes such that $p < q$, then $\mathcal{B}(G)=K_2 \sqcup K_{1, p^2-1} \sqcup K_{1, q^2-1} \sqcup K_{1, p^2q^2-p^2-q^2+1}$.
\end{theorem}
\begin{proof}
Let $G = \langle a \rangle$. Then $V(\mathcal{B}(G))=G \times G \sqcup \{\{1\}, \langle a^q \rangle, \langle a^p \rangle, \langle a \rangle\}$ as there are exactly four subgroups of $G$ and they are cyclic. Here $|\langle a^q \rangle|=p, |\langle a^p \rangle|=q$ \, and \, $|\langle a \rangle|=pq$. As such, \, $\mathcal{B}(G)[\{\langle a^q \rangle\} \sqcup $ $\Nbd_{\mathcal{B}(G)}(\langle a^q \rangle)] =K_{1, p^2-1}$ and $\mathcal{B}(G)[\{\langle a^p \rangle\} \sqcup \Nbd_{\mathcal{B}(G)}(\langle a^p \rangle)] =K_{1, q^2-1}$. Thus, $|\Nbd_{\mathcal{B}(G)}(\langle a \rangle)|=p^2q^2-(p^2+q^2-2+1)$ and $\mathcal{B}(G)[\{\langle a \rangle\} \sqcup \Nbd_{\mathcal{B}(G)}(\langle a \rangle)] = K_{1, p^2q^2-p^2-q^2+1}$. Therefore,
	\begin{align*}
		\mathcal{B}(G)&= \underset{H \in L(G)}{\sqcup} \mathcal{B}(G)[H \sqcup \Nbd_{\mathcal{B}(G)}(H)] \\
		&= K_2 \sqcup K_{1, p^2-1} \sqcup K_{1, q^2-1} \sqcup K_{1, p^2q^2-p^2-q^2+1}.
	\end{align*}
	This completes the proof.
\end{proof}
\begin{theorem}\label{structure_cyclic_p^2q}
If $G$ is a cyclic group of order $p^2q$, where $p$, $q$ are two distinct primes, then $\mathcal{B}(G)=K_2 \sqcup K_{1, p^2-1} \sqcup K_{1, p^4-p^2} \sqcup K_{1, q^2-1} \sqcup K_{1, p^2q^2-p^2-q^2+1} \sqcup K_{1, p^4q^2-p^2q^2-p^4+p^2}$.
\end{theorem}
\begin{proof}
Let $G = \langle a \rangle$. Then $V(\mathcal{B}(G))=G \times G \sqcup \{\{1\}, \langle a^{pq} \rangle, \langle a^q \rangle, \langle a^{p^2} \rangle, \langle a^p \rangle, \langle a \rangle\}$ as there are exactly six subgroups of $G$ and they are cyclic. Here $|\langle a^{pq} \rangle|=p, |\langle a^q \rangle|=p^2, |\langle a^{p^2} \rangle|=q, |\langle a^p \rangle|=pq$ \, and \, $|\langle a \rangle|=p^2q$. As such, \, $\mathcal{B}(G)[\{\langle a^{pq} \rangle\} \sqcup \Nbd_{\mathcal{B}(G)}(\langle a^{pq} \rangle)] =K_{1, p^2-1}, \mathcal{B}(G)[\{\langle a^q \rangle\} \sqcup \Nbd_{\mathcal{B}(G)}(\langle a^q \rangle)] =K_{1, p^4-p^2}, \mathcal{B}(G)[\{\langle a^{p^2} \rangle\} \sqcup \Nbd_{\mathcal{B}(G)}(\langle a^{p^2} \rangle)] =K_{1, q^2-1}$ and $\mathcal{B}(G)[\{\langle a^p \rangle\} \sqcup \Nbd_{\mathcal{B}(G)}(\langle a^p \rangle)] =K_{1, p^2q^2-p^2-q^2+1}$. Thus, $|\Nbd_{\mathcal{B}(G)}(\langle a \rangle)|=p^4q^2-(p^2+p^4-p^2+q^2+p^2q^2-p^2-q^2-2+2)$ and $\mathcal{B}(G)[\{\langle a \rangle\} \sqcup \Nbd_{\mathcal{B}(G)}(\langle a \rangle)] =K_{1, p^4q^2-p^2q^2-p^4+p^2}$. Therefore,
	\begin{align*}
		\mathcal{B}(G)&= \underset{H \in L(G)}{\sqcup} \mathcal{B}(G)[H \sqcup \Nbd_{\mathcal{B}(G)}(H)] \\
		&= K_2 \sqcup K_{1, p^2-1} \sqcup K_{1, p^4-p^2} \sqcup K_{1, q^2-1} \sqcup K_{1, p^2q^2-p^2-q^2+1} \sqcup K_{1, p^4q^2-p^2q^2-p^4+p^2}.
	\end{align*}
	This completes the proof.
\end{proof}
\begin{theorem}\label{structure_cyclic_p^2q^2}
If $G$ is a cyclic group of order $p^2q^2$, where $p$, $q$ are two primes such that $p < q$, then $\mathcal{B}(G)=K_2 \sqcup K_{1, p^2-1} \sqcup K_{1, p^4-p^2} \sqcup K_{1, q^2-1} \sqcup K_{1, q^4-q^2} \sqcup K_{1, p^2q^2-p^2-q^2+1} \sqcup K_{1, p^2q^4-p^2q^2-q^4+q^2} \sqcup K_{1, p^4q^2-p^2q^2-p^4+p^2} \sqcup K_{1, p^4q^4-p^2q^4-p^4q^2+p^2q^2}$.
\end{theorem}
\begin{proof}
Let $G = \langle a \rangle$. Then $V(\mathcal{B}(G))=G \times G \sqcup \{\{1\}, \langle a^{pq^2} \rangle, \langle a^{q^2} \rangle, \langle a^{p^2q} \rangle, \langle a^{p^2} \rangle, \langle a^{pq} \rangle, \langle a^p \rangle,$  $\langle a^q \rangle, \langle a \rangle\}$ as there are exactly nine subgroups of $G$ and they are cyclic. Here $|\langle a^{pq^2} \rangle|=p$, $|\langle a^{q^2} \rangle|=p^2$, $|\langle a^{p^2q} \rangle|=q$, $|\langle a^{p^2} \rangle|=q^2$, $|\langle a^{pq} \rangle|=pq$, $|\langle a^p \rangle|=pq^2$, $|\langle a^q \rangle|=p^2q$ \, and \, $|\langle a \rangle|=p^2q^2$. As such, \, 

  $\mathcal{B}(G)[\{\langle a^{pq^2} \rangle\} \sqcup \Nbd_{\mathcal{B}(G)}(\langle a^{pq^2} \rangle)] =K_{1, p^2-1}$, \, $\mathcal{B}(G)[\{\langle a^{q^2} \rangle\}$ $\sqcup \Nbd_{\mathcal{B}(G)}(\langle a^{q^2} \rangle)] =K_{1, p^4-p^2}$, \,

  $\mathcal{B}(G)[\{\langle a^{p^2q} \rangle\} \sqcup \Nbd_{\mathcal{B}(G)}(\langle a^{p^2q} \rangle)] =K_{1, q^2-1}$, \, $\mathcal{B}(G)[\{\langle a^{p^2} \rangle\}$ $ \sqcup \Nbd_{\mathcal{B}(G)}(\langle a^{p^2} \rangle)] \, = \, K_{1, q^4-q^2}$, \, \quad
 
\noindent  $\mathcal{B}(G)[\{\langle a^{pq} \rangle\} \sqcup \Nbd_{\mathcal{B}(G)}(\langle a^{pq} \rangle)] \, = \, K_{1, p^2q^2-p^2-q^2+1}$, $\mathcal{B}(G)[\{\langle a^{p} \rangle\} \sqcup \Nbd_{\mathcal{B}(G)}(\langle a^{p} \rangle)] =K_{1, p^2q^4-p^2q^2-q^4+q^2}$ \, and \, $\mathcal{B}(G)[\{\langle a^{q} \rangle\} \, \sqcup \, \Nbd_{\mathcal{B}(G)}(\langle a^{q} \rangle)]$ $ =K_{1, p^4q^2-p^2q^2-p^4+p^2}$. 

Thus, \, $|\Nbd_{\mathcal{B}(G)}(\langle a \rangle)|=p^4q^4-(p^2+p^4-p^2+q^2+q^4-q^2+p^2q^2-p^2-q^2+p^2q^4-p^2q^2-q^4+q^2+p^4q^2-p^2q^2-p^4+p^2-2+2)$ and $\mathcal{B}(G)[\{\langle a \rangle\} \sqcup \Nbd_{\mathcal{B}(G)}(\langle a \rangle)] =K_{1, p^4q^4-p^2q^4-p^4q^2+p^2q^2}$. Therefore,
	\begin{align*}
		\mathcal{B}(G)&= \underset{H \in L(G)}{\sqcup} \mathcal{B}(G)[H \sqcup \Nbd_{\mathcal{B}(G)}(H)] \\
		&= K_2 \sqcup K_{1, p^2-1} \sqcup K_{1, p^4-p^2} \sqcup K_{1, q^2-1} \sqcup K_{1, q^4-q^2} \sqcup K_{1, p^2q^2-p^2-q^2+1}  \\ 
           & \qquad \qquad \sqcup K_{1, p^2q^4-p^2q^2-q^4+q^2} \sqcup K_{1, p^4q^2-p^2q^2-p^4+p^2} \sqcup K_{1, p^4q^4-p^2q^4-p^4q^2+p^2q^2}.
	\end{align*}
	This completes the proof.
\end{proof}

%We conclude this section noting that   $\mathcal{B}(G)$ is a union of   star graphs and number of components in $\mathcal{B}(G)$ is $|L(G)|$. For each subgroup $H$ of $G$ we have $\mathcal{B}(G)[H \sqcup \Nbd(H)] = K_{1, m}$, for some $m$.  

\section{Topological indices of $\mathcal{B}(G)$}
 Let $\Gamma$ be the set of all graphs. A topological index is a function $T :  \Gamma \to \mathbb{R}$ such that $T(\mathcal{G}_1) = T(\mathcal{G}_2)$ whenever the graphs $\mathcal{G}_1$ and $\mathcal{G}_2$ are isomorphic. These numerical invariants are used to obtain information about the physical and chemical properties of some molecules. Various topological indices have been defined using different parameters of graphs since 1947. Some of them have proved to be useful also in other areas where connectivity patterns play an important role, such as interprocessor connections and complex networks. A couple of such popular degree-based topological indices are the Zagreb indices. These were introduced by Gutman and Trinajsti{\'c} \cite{Gut-Trin-72} in 1972, and are used in examining the dependence of total $\pi$-electron energy on molecular structure. Zagreb indices are also used in studying molecular complexity, chirality, ZE-isomerism and heterosystems etc. \cite{Z-index-30y-2003}. One can find a survey on the mathematical properties of Zagreb indices in \cite{Gut-Das-2004}. Zagreb indices of commuting and non-commuting graphs of finite non-abelian groups have been studied in \cite{DSN-23, mirzargar2012some}. Zagreb indices of super commuting graphs and SGB-graphs of finite groups have been computed in \cite{DN-24} and \cite{DEN-24} respectively. 

  Let $\mathcal{G}$ be a simple undirected graph with vertex set $V(\mathcal{G})$ and edge set $e(\mathcal{G})$. The first and second Zagreb indices of $\mathcal{G}$, denoted by $M_{1}(\mathcal{G})$ and $M_{2}(\mathcal{G})$ respectively, are defined as 
\[
M_{1}(\mathcal{G}) = \sum\limits_{v \in V(\mathcal{G})} \deg(v)^{2}  \text{ and }  M_{2}(\mathcal{G}) = \sum\limits_{uv \in e(\mathcal{G})} \deg(u)\deg(v).
\]
In 2007, Hansen and Vuki{\v{c}}evi{\'c} posed the following conjecture to show a comparison between the first and second Zagreb indices \cite{hansen2007comparing}.
\begin{conj}\label{Conj}
	(Hansen-Vuki{\v{c}}evi{\'c} Conjecture) For any simple finite graph $\mathcal{G}$, 
	\begin{equation}\label{Conj-eq}
		\dfrac{M_{2}(\mathcal{G})}{\vert e(\mathcal{G}) \vert} \geq \dfrac{M_{1}(\mathcal{G})}{\vert V(\mathcal{G}) \vert} .
	\end{equation}
\end{conj}
In the same paper, they disproved the conjecture by showing that it is not true  if $\Gamma = K_{1, 5} \sqcup K_3$ \cite{hansen2007comparing}. However, Hansen and Vuki{\v{c}}evi{\'c} \cite{hansen2007comparing}  showed that  Conjecture \ref{Conj}  holds for chemical graphs. 
%In \cite{vukicevic2007comparing}, it was shown that the conjecture holds for trees with equality in \eqref{Conj-eq} when $\Gamma$ is a star graph. In \cite{liu2008conjecture}, it was shown that the conjecture holds for connected unicyclic graphs with equality when the graph is a cycle. %The case when equality holds in \eqref{Conj-eq} is studied extensively in \cite{vukicevic2011some}.
Upon further studies, it was shown that the conjecture holds for trees and with equality in \eqref{Conj-eq} when $\Gamma$ is a star graph \cite{vukicevic2007comparing}. Thereafter, it was shown that the conjecture holds for connected unicyclic graphs with equality when the graph is a cycle (see \cite{liu2008conjecture}). The case when equality holds in \eqref{Conj-eq} is studied extensively in \cite{vukicevic2011some}. Das et al. \cite{DSN-23} have obtained various finite non-abelian groups such that their commuting/non-commuting graphs satisfy Hansen-Vuki{\v{c}}evi{\'c} Conjecture. Several finite groups such that their super commuting graphs satisfy Hansen-Vuki{\v{c}}evi{\'c} Conjecture have also been obtained by Das and Nath \cite{DN-24} recently. Zagreb indices of commuting conjugacy class graph have been computed and verified  Conjecture \ref{Conj} in \cite{Das-Nath-2023} for the classes of finite groups considered in \cite{Salah-2020, SA-2020, SA-CA-2020}. 
A survey on comparing Zagreb indices can be found in \cite{Liu-You-2011}. 

 In this section, we obtain expressions for first and second Zagreb indices of $\mathcal{B}(G)$ to finally show that $\mathcal{B}(G)$ satisfying Hansen-Vuki{\v{c}}evi{\'c} conjecture for the groups considered in Section 2. The following lemma is useful in our computations.
%We  also check if $\mathcal{B}(D_{2p}), \mathcal{B}(D_{2p^2}), \mathcal{B}(Q_{4p})$ and $\mathcal{B}(Q_{4p^2})$ satisfy Hansen-Vuki{\v{c}}evi{\'c} conjecture. We begin with the following result.
\begin{lemma}\label{Zagreb_indices_B(G)}
   Let $G$ be a finite group. Then the Zagreb indices of $\mathcal{B}(G)$ are given by
	\[
	M_1(\mathcal{B}(G))=|G|^2+\sum_{H \in L(G)}\left(\deg_{\mathcal{B}(G)}(H)\right)^2 \quad \text{and} \quad M_2(\mathcal{B}(G))= M_1(\mathcal{B}(G))-|G|^2.
	\]
 Further, $\mathcal{B}(G)$ satisfies Hansen-Vuki{\v{c}}evi{\'c} conjecture if and only if $|L(G)|M_2(\mathcal{B}(G))-|G|^4 \geq 0$.
\end{lemma}
\begin{proof}
We have 
\[M_{1}(\mathcal{B}(G)) = \sum_{v \in V(\mathcal{B}(G))} \deg(v)^{2} =\sum_{(a, b) \in G \times G}\left(\deg_{\mathcal{B}(G)}((a, b))\right)^2+\sum_{H \in L(G)}\left(\deg_{\mathcal{B}(G)}(H)\right)^2.\]
From the definition of $\mathcal{B}(G)$, we have $\deg_{\mathcal{B}(G)}((a, b))=1$. As such, 
\[M_{1}(\mathcal{B}(G)) =\sum_{(a, b) \in G \times G}1+\sum_{H \in L(G)}\left(\deg_{\mathcal{B}(G)}(H)\right)^2=|G|^2+\sum_{H \in L(G)}\left(\deg_{\mathcal{B}(G)}(H)\right)^2.\]
Also, 
\begin{align*}
		M_2&(\mathcal{B}(G))= \sum_{uv \in e(\mathcal{B}(G))} \deg(u)\deg(v)\\
		&\quad=\sum_{(a, b)H \in e(\mathcal{B}(G))}\deg_{\mathcal{B}(G)}((a, b))\deg_{\mathcal{B}(G)}(H)
			= \sum_{(a, b)H \in e(\mathcal{B}(G))} \deg_{\mathcal{B}(G)}(H). 
\end{align*}
Since, for each $H \in L(G)$, $\deg_{\mathcal{B}(G)}(H)$ appears $\deg_{\mathcal{B}(G)}(H)$ many times in the above sum we have
\[
M_2(\mathcal{B}(G))= \sum_{H \in L(G)} \left(\deg_{\mathcal{B}(G)}(H)\right)^2=M_1(\mathcal{B}(G))-|G|^2.
\]
Now, $\vert V(\mathcal{B}(G)) \vert =|G|^2+|L(G)|$ and $\vert e(\mathcal{B}(G))\vert =|G|^2$ (see Theorem 3.2 of \cite{DEN-23}). As such,
\begin{align*}
		\frac{M_{2}(\mathcal{B}(G))}{\vert e(\mathcal{B}(G))\vert} - \frac{M_{1}(\mathcal{B}(G))}{\vert V(\mathcal{B}(G)) \vert}&= \frac{M_2(\mathcal{B}(G))}{|G|^2}-\frac{M_{2}(\mathcal{B}(G))+|G|^2}{|G|^2+|L(G)|}= \frac{|L(G)|M_2(\mathcal{B}(G))-|G|^4}{|G|^2(|G|^2+|L(G)|)}.
	\end{align*}
Hence, the result follows noting that $|G|^2(|G|^2+|L(G)|) > 0$.
%= \sum_{(a, b) \in G \times G} 1 + \sum_{H \in L(G)}\left(|H|^2\varphi_2(H)\right)^2 
\end{proof}
%\begin{lemma}\label{condition_for_satisfying_conjecture} 
	%\cite{DEN-24} Let $G$ be a finite group. Then \quad $\frac{M_{2}(\mathcal{B}(G))}{\vert e(\mathcal{B}(G))\vert} \geq \frac{M_{1}(\mathcal{B}(G))}{\vert V(\mathcal{B}(G)) \vert} $ if and only if  
%\end{lemma}
\begin{theorem}\label{cyclic_pq}
If $G$ is a cyclic group of order $pq$, where $p$ and $q$ are two  primes such that $p < q$, then $M_1(\mathcal{B}(G))=p^4q^4-2p^4q^2-2p^2q^4+5p^2q^2+2p^4+2q^4-4p^2-4q^2+4$ and $M_2(\mathcal{B}(G))=p^4q^4-2p^4q^2-2p^2q^4+4p^2q^2+2p^4+2q^4-4p^2-4q^2+4$.	Further, $\frac{M_{2}(\mathcal{B}(G))}{\vert e(\mathcal{B}(G))\vert} > \frac{M_{1}(\mathcal{B}(G))}{\vert V(\mathcal{B}(G)) \vert}$.
\end{theorem}
\begin{proof}
	By Theorem \ref{structure_cyclic_pq}, we have $\mathcal{B}(G)=K_2 \sqcup K_{1, p^2-1} \sqcup K_{1, q^2-1} \sqcup K_{1, p^2q^2-p^2-q^2+1}$. As such $|L(G)| = 4$ and by Lemma \ref{Zagreb_indices_B(G)}, we have
\begin{align*}
M_1(\mathcal{B}(G))&=|G|^2+\sum_{H \in L(G)}\left(\deg_{\mathcal{B}(G)}(H)\right)^2 \\
&=p^4q^4-2p^4q^2-2p^2q^4+5p^2q^2+2p^4+2q^4-4p^2-4q^2+4
\end{align*} 
and
\begin{align*}
M_2(\mathcal{B}(G))&=M_1(\mathcal{B}(G))-|G|^2 \\
&= p^4q^4-2p^4q^2-2p^2q^4+4p^2q^2+2p^4+2q^4-4p^2-4q^2+4.
\end{align*}
Further, 
\begin{align*}
|L(G)|&M_2(\mathcal{B}(G))-|G|^4 \\
&= 4(p^4q^4-2p^4q^2-2p^2q^4+4p^2q^2+2p^4+2q^4-4p^2-4q^2+4)-p^4q^4 \\
&= p^2q^2(p^2(3q^2-8)-8q^2)+16p^2q^2+8p^2(p^2-2)+8q^2(q^2-2)+16. 
\end{align*}
We have $p^2-2>0$ and $q^2-2>0$ for all $p, q$. In fact, $p^2 \geq 4$ and $q^2 \geq 9$.  Therefore,  $3q^2-8 > 2q^2$ and so $p^2(3q^2-8)-8q^2 > 0$.
%
%\noindent \textbf{Case 1.} If $p > q$ then $p^2>2q^2$ and $3q^2-8 \geq 4$ for all $q^2 \geq 4$.
%
%\noindent \textbf{Case 2.} If $p<q$ then $p^2 \geq 4$ and $3q^2-8 > 2q^2$ for all $q^2 \geq 9$.
%
  Thus,  $|L(G)|M_2(\mathcal{B}(G))-|G|^4 > 0$ and hence the result follows from Lemma \ref{Zagreb_indices_B(G)}.
\end{proof}
\begin{theorem}\label{cyclic_p^2q}
    If $G$ is a cyclic group of order $p^2q$, where $p$, $q$ are two distinct primes, then $M_1(\mathcal{B}(G))=p^8q^4+2p^4q^4+4p^6q^2-2p^6q^4+2p^8+4p^4-2p^8q^2-3p^4q^2-4p^6-2p^2q^4+4p^2q^2+2q^4-4p^2-4q^2+4$ and $M_2(\mathcal{B}(G))=p^8q^4+2p^4q^4+4p^6q^2-2p^6q^4+2p^8+4p^4-2p^8q^2-4p^4q^2-4p^6-2p^2q^4+4p^2q^2+2q^4-4p^2-4q^2+4$. Further, $\frac{M_{2}(\mathcal{B}(G))}{\vert e(\mathcal{B}(G))\vert} > \frac{M_{1}(\mathcal{B}(G))}{\vert V(\mathcal{B}(G)) \vert}$.
\end{theorem}
\begin{proof}
    By Theorem \ref{structure_cyclic_p^2q}, we have
   \[  
     \mathcal{B}(G)=K_2 \sqcup K_{1, p^2-1} \sqcup K_{1, p^4-p^2} \sqcup K_{1, q^2-1} \sqcup K_{1, p^2q^2-p^2-q^2+1} \sqcup K_{1, p^4q^2-p^2q^2-p^4+p^2}.
     \] 
     As such $|L(G)| = 6$ and by Lemma \ref{Zagreb_indices_B(G)}, we have
\begin{align*}
M_1(\mathcal{B}(G))&=|G|^2+\sum_{H \in L(G)}\left(\deg_{\mathcal{B}(G)}(H)\right)^2 \\
&=p^8q^4+2p^4q^4+4p^6q^2-2p^6q^4+2p^8+4p^4-2p^8q^2-3p^4q^2-4p^6-2p^2q^4 \\ & \qquad \qquad \qquad \qquad \qquad \qquad \qquad \qquad \qquad +4p^2q^2+2q^4-4p^2-4q^2+4
\end{align*}
and
\begin{align*}
M_2(\mathcal{B}(G))&=M_1(\mathcal{B}(G))-|G|^2 \\
&=p^8q^4+2p^4q^4+4p^6q^2-2p^6q^4+2p^8+4p^4-2p^8q^2-4p^4q^2-4p^6-2p^2q^4 \\
& \qquad \qquad \qquad \qquad \qquad \qquad \qquad \qquad \qquad \, \, \, +4p^2q^2+2q^4-4p^2-4q^2+4.
\end{align*}
Further, 
\begin{align*}
	&|L(G)|M_2(\mathcal{B}(G))-|G|^4 \\
    &= p^6q^2(5p^2q^2-12q^2-12p^2)+12p^2q^2(2p^4-2p^2)+12p^2q^4(p^2-1)+12p^6(p^2-2) \\
    & \qquad \qquad \qquad \qquad \qquad \qquad \qquad \qquad+24p^2(q^2-1)+12q^2(q^2-2)+24. 
	\end{align*}
We have $p^2-1$, $q^2-1$, $p^2-2$, $q^2-2$ and $2p^4-2p^2$ are all positive for all $p, q$. If $p > q$ then  $p^2 \geq 9$ and so $5p^2-12 > 3p^2$. Therefore, $q^2(5p^2-12) > 3q^2p^2 > 12p^2$ since  $q^2 \geq 4$ and so $5p^2q^2-12q^2-12p^2 = q^2(5p^2-12) - 12p^2 > 0$. If $p < q$ then  $q^2 \geq 9$ and so $5q^2-12 > 3q^2$. Therefore, $p^2(5q^2-12) > 3q^2p^2 > 12q^2$ since  $p^2 \geq 4$ and so $5p^2q^2-12q^2-12p^2 = p^2(5q^2-12) - 12q^2 > 0$. 
%\noindent \textbf{Case 1.} If $p > q$ then $q^2 \geq 4$ and $5p^2-12 > 3p^2$ for all $p^2 \geq 9$.
%\noindent \textbf{Case 2.} If $p<q$ then $q^2>2p^2$ and $5p^2-12 >6$ for all $p^2 \geq 4$.
 Thus,  $|L(G)|M_2(\mathcal{B}(G))-|G|^4 > 0$ and hence the result follows from Lemma \ref{Zagreb_indices_B(G)}.
\end{proof}
\begin{theorem}\label{cyclic_p^2q^2}
    If $G$ is a cyclic group of order $p^2q^2$, where $p$ and $q$ are two distinct primes such that $p < q$, then $M_1(\mathcal{B}(G))=p^8q^8-2p^6q^8-2p^8q^6-4p^4q^6-4p^6q^4-2p^2q^8-2p^8q^2+2p^4q^8+2p^8q^4+4p^6q^6+5p^4q^4+4p^2q^6+4p^6q^2-4p^2q^4-4p^4q^2+4p^2q^2+2q^8+2p^8-4q^6-4p^6+4q^4+4p^4-4p^2-4q^2+4$ and $M_2(\mathcal{B}(G))=p^8q^8-2p^6q^8-2p^8q^6-4p^4q^6-4p^6q^4-2p^2q^8-2p^8q^2+2p^4q^8+2p^8q^4+4p^6q^6+4p^4q^4+4p^2q^6+4p^6q^2-4p^2q^4-4p^4q^2+4p^2q^2+2q^8+2p^8-4q^6-4p^6+4q^4+4p^4-4p^2-4q^2+4$. Further, $\frac{M_{2}(\mathcal{B}(G))}{\vert e(\mathcal{B}(G))\vert} > \frac{M_{1}(\mathcal{B}(G))}{\vert V(\mathcal{B}(G)) \vert}$.
\end{theorem}
\begin{proof}
    By Theorem \ref{structure_cyclic_p^2q^2}, we have $\mathcal{B}(G)=K_2 \sqcup K_{1, p^2-1} \sqcup K_{1, p^4-p^2} \sqcup K_{1, q^2-1} \sqcup K_{1, q^4-q^2} \sqcup K_{1, p^2q^2-p^2-q^2+1} \sqcup K_{1, p^2q^4-p^2q^2-q^4+q^2} \sqcup K_{1, p^4q^2-p^2q^2-p^4+p^2} \sqcup K_{1, p^4q^4-p^2q^4-p^4q^2+p^2q^2}$. As such $|L(G)| = 9$ and by Lemma \ref{Zagreb_indices_B(G)}, we have
\begin{align*}
& \qquad M_1(\mathcal{B}(G))=|G|^2+\sum_{H \in L(G)}\left(\deg_{\mathcal{B}(G)}(H)\right)^2 \\
&=p^8q^8-2p^6q^8-2p^8q^6-4p^4q^6-4p^6q^4-2p^2q^8-2p^8q^2+2p^4q^8+2p^8q^4+4p^6q^6\\
& \qquad \qquad  +5p^4q^4+4p^2q^6+4p^6q^2-4p^2q^4-4p^4q^2+4p^2q^2+2q^8+2p^8-4q^6-4p^6 \\
& \qquad \qquad \qquad \qquad \qquad \qquad \qquad \qquad \qquad \qquad \qquad \quad +4q^4+4p^4-4p^2-4q^2+4
\end{align*}
and
\begin{align*}
& \qquad M_2(\mathcal{B}(G))=M_1(\mathcal{B}(G))-|G|^2 \\
&=p^8q^8-2p^6q^8-2p^8q^6-4p^4q^6-4p^6q^4-2p^2q^8-2p^8q^2+2p^4q^8+2p^8q^4+4p^6q^6+4p^4q^4 \\
& \qquad \qquad \qquad \qquad \quad +4p^2q^6+4p^6q^2-4p^2q^4-4p^4q^2+4p^2q^2+2q^8+2p^8-4q^6-4p^6 \\
& \qquad \qquad \qquad \qquad \qquad \qquad \qquad \qquad \qquad \qquad \qquad \qquad +4q^4+4p^4-4p^2-4q^2+4.
\end{align*}
Further, 
	\begin{align*}
		&|L(G)|M_2(\mathcal{B}(G))-|G|^4 \\
            &= 2p^6q^6(4p^2q^2-9q^2-9p^2+18)+18p^2q^6(p^2q^2-q^2-2p^2+2) \\
            & \qquad+18p^6q^2(p^2q^2-p^2-2q^2+2)+36p^2q^2(p^2q^2-p^2-q^2)+36p^2q^2+18q^6(q^2-2) \\
            & \qquad \qquad \qquad \qquad \qquad \qquad \qquad +18p^6(p^2-2)+36q^2(q^2-1)+36p^2(p^2-1)+36. 
	\end{align*}
We have $p^2-1$, $q^2-1$, $p^2-2$ and $q^2-2$  are all positive for all $p, q$. If $p < q$ then $q^2 \geq 9$ and so $4q^2 - 9 \geq 3q^2$. Therefore, $p^2(4q^2 - 9) \geq 3p^2q^2 > 12q^2$ since $p^2 \geq 4$ and so $4p^2q^2-9q^2-9p^2 + 18= p^2(4q^2 - 9) - 9q^2 + 18 \geq 12q^2 - 9q^2 + 18 = 3q^2 + 18> 0$. We have $p^2q^2 - p^2 - 2q^2 = q^2(p^2 - 2) - p^2 > 0$ since $q^2 > p^2$ and $p^2 - 2 > 1$. Therefore, $p^2q^2 - p^2 - 2q^2 + 2 > 0$ and $p^2q^2 - p^2 - q^2 > p^2q^2 - p^2 - 2q^2 > 0$. Also, 
$p^2q^2 -  q^2 - 2p^2  = q^2(p^2 - 1) - 2p^2 > 0$ since $q^2 > p^2$ and $p^2 - 1 > 2$. Therefore, $p^2q^2 - q^2 - 2p^2 + 2 > 0$.  
% We have $p^2, q^2>2$ for all $p, q$.
%
%\noindent \textbf{Case 1.} If $p > q$ then $p^2 > 2q^2$ and $4q^2-9 > 5$ for all $q^2 \geq 4$.
%
%\noindent \textbf{Case 2.} If $p<q$ then $p^2 \geq 4$, $q^2-1 > \frac{q^2}{4}$ and $4q^2-9 \geq 3q^2$ for all $q^2 \geq 9$.
%
 Thus,  $|L(G)|M_2(\mathcal{B}(G))-|G|^4 > 0$ and hence the result follows from Lemma \ref{Zagreb_indices_B(G)}.
\end{proof}
In view of Theorems \ref{cyclic_pq}--\ref{cyclic_p^2q^2}, we conclude that 
 $\mathcal{B}(G)$ satisfies Hansen-Vuki{\v{c}}evi{\'c} conjecture if $G$ is isomorphic to a cyclic group of order $pq, p^2q$ and $p^2q^2$ for any two distinct  primes $p$ and $q$. 
%It may be interesting to conclude that $\mathcal{B}(G)$ satisfies Hansen-Vuki{\v{c}}evi{\'c} conjecture  for any finite group $G$.

\section{Other topological indices}
Soon after the introduction of Zagreb indices, various other degree-based topological indices were also introduced and studied. Randic Connectivity index, Atom-Bond Connectivity index, Geometric-Arithmetic index, Harmonic index, Sum-Connectivity index etc. are some of the popular ones. The Randic Connectivity index $R(\mathcal{G})$, Atom-Bond Connectivity index $\ABC(\mathcal{G})$, Geometric-Arithmetic index $\GA(\mathcal{G})$, Harmonic index $H(\mathcal{G})$ and Sum-Connectivity index  $\SCI(\mathcal{G})$ of $\mathcal{G}$  are defined as
\[
R(\mathcal{G})=\sum_{uv \in e(\mathcal{G})}\left(\deg(u)\deg(v)\right)^{\frac{-1}{2}}, \quad \ABC(\mathcal{G})=\sum_{uv\in e(\mathcal{G})}\left(\frac{\deg(u)+\deg(v)-2}{\deg(u)\deg(v)}\right)^{\frac{1}{2}}, 
\]
\[
\GA(\mathcal{G})=\sum_{uv\in e(\mathcal{G})} \frac{\sqrt{\deg(u)\deg(v)}}{\frac{1}{2}(\deg(u)+\deg(v))}, \quad H(\mathcal{G})=\sum_{uv\in e(\mathcal{G})}\frac{2}{\deg(u)+\deg(v)} 
\]
and
\[
\SCI(\mathcal{G})=\sum_{uv \in e(\mathcal{G})} \left(\deg(u)+\deg(v)\right)^{\frac{-1}{2}}.
\]
In this section, we obtain these topological indices of $\mathcal{B}(G)$ for the groups considered in Section 2. The following result is useful in our computations.
%We use the following result :
\begin{lemma}\label{other_topological_index}
\cite[Lemma 4.1]{DEN-24} For any finite group $G$ we have 
  
\noindent  $R(\mathcal{B}(G))=\!\!\sum\limits_{H \in L(G)}\left(\deg_{\mathcal{B}(G)}(H)\right)^{\frac{1}{2}}$,
$\ABC(\mathcal{B}(G))=\!\!\sum\limits_{H \in L(G)}\left(\left(\deg_{\mathcal{B}(G)}(H)\right)^2-\deg_{\mathcal{B}(G)}(H)\right)^{\frac{1}{2}}$,
		 
$\GA(\mathcal{B}(G))=\sum\limits_{H \in L(G)} \frac{2\left(\deg_{\mathcal{B}(G)}(H)\right)^{\frac{3}{2}}}{(1+\deg_{\mathcal{B}(G)}(H))}$, \qquad
		 $H(\mathcal{B}(G)=\sum\limits_{H \in L(G)}\frac{2\deg_{\mathcal{B}(G)}(H)}{1+\deg_{\mathcal{B}(G)}(H)}$
		 
and	\quad	 $\SCI(\mathcal{B}(G))=\sum\limits_{H \in L(G)} \left(1+\deg_{\mathcal{B}(G)}(H)\right)^{\frac{-1}{2}}\deg_{\mathcal{B}(G)}(H)$.
\end{lemma}
In the computation of Zagreb indices of $\mathcal{B}(G)$ for the cyclic groups of order $pq$, $p^2q$ and $p^2q^2$  we have computed $\deg_{\mathcal{B}(G)}(H)$ for all $H \in L(G)$ which can be seen in the proofs of Theorem \ref{cyclic_pq} -- Theorem \ref{cyclic_p^2q^2}. 
%Using those degrees of $H$ in $\mathcal{B}(G)$  and 
Therefore, by Lemma \ref{other_topological_index} we get the following theorems.
\begin{theorem}
If $G$ is a cyclic group of order $pq$, where $p$ and $q$ are two  primes such that $p < q$, then 
 
\begin{center}
$R(\mathcal{B}(G))=1+\sqrt{p^2-1}(1+\sqrt{q^2-1})+\sqrt{q^2-1}$, 

$\ABC(\mathcal{B}(G))=\sqrt{p^2-1}(\sqrt{p^2-2}+\sqrt{(q^2-1)(p^2q^2-p^2-q^2)})+\sqrt{(q^2-1)(q^2-2)}$,
	 
	  $\GA(\mathcal{B}(G))=1+2\sqrt{(p^2-1)^3}\left(\frac{1}{p^2}+\frac{\sqrt{(q^2-1)^3}}{p^2q^2-p^2-q^2+2}\right)+\frac{2\sqrt{(q^2-1)^3}}{q^2}$,

$H(\mathcal{B}(G))=1+2(p^2-1)\left(\frac{1}{p^2}+\frac{q^2-1}{p^2q^2-p^2-q^2+2}\right)+\frac{2(q^2-1)}{q^2}$
\text{ and }
$\SCI(\mathcal{B}(G))=\frac{1}{\sqrt{2}}+(p^2-1)\left(\frac{1}{p}+\frac{q^2-1}{\sqrt{p^2q^2-p^2-q^2+2}}\right)+\frac{q^2-1}{q}$.  
   \end{center} 
\end{theorem}

\begin{theorem}
If $G$ is a cyclic group of order $p^2q$, where $p$ and $q$ are two distinct primes, then 
 
\begin{center}
$R(\mathcal{B}(G))=1+\sqrt{p^2-1}(p+1)(1+\sqrt{q^2-1})+\sqrt{q^2-1}$, 

$\ABC(\mathcal{B}(G))=\sqrt{p^2-1}(\sqrt{p^2-2}+p\sqrt{p^4-p^2-1}+\sqrt{(q^2-1)(p^2q^2-p^2-q^2)}+p\sqrt{(q^2-1)(p^4q^2-p^2q^2-p^4+p^2-1}))+\sqrt{(q^2-1)(q^2-2)}$,
 \end{center} 	 
{\small{ 	  $\GA(\mathcal{B}(G))=1+2\sqrt{(p^2-1)^3}\left(\frac{1}{p^2}+\frac{p^3}{p^4-p^2+1}+\frac{\sqrt{(q^2-1)^3}}{p^2q^2-p^2-q^2+2}+\frac{p^3\sqrt{(q^2-1)^3}}{p^4q^2-p^2q^2-p^4+p^2+1}\right)+\frac{2\sqrt{(q^2-1)^3}}{q^2}$}},
\begin{center}
$H(\mathcal{B}(G))=1+2(p^2-1)\left(\frac{1}{p^2}+\frac{p^2}{p^4-p^2+1}+\frac{q^2-1}{p^2q^2-p^2-q^2+2}+\frac{p^2(q^2-1)}{p^4q^2-p^2q^2-p^4+p^2+1}\right)+\frac{2(q^2-1)}{q^2}$
\text{ and }
$\SCI(\mathcal{B}(G))=\frac{1}{\sqrt{2}}+(p^2-1)\left(\frac{1}{p}+\frac{p^2}{\sqrt{p^4-p^2+1}}+\frac{q^2-1}{\sqrt{p^2q^2-p^2-q^2+2}}+\frac{p^2(q^2-1)}{p^4q^2-p^2q^2-p^4+p^2+2}\right)+\frac{q^2-1}{q}$.  
   \end{center} 
\end{theorem}
\begin{theorem}
If $G$ is a cyclic group of order $p^2q^2$, where $p$ and $q$ are two  primes such that $p<q$, then 
 
\begin{align*}
    R(\mathcal{B}(G))=1+\sqrt{p^2-1}(p+1+\sqrt{q^2-1}(1+p+q+pq))+\sqrt{q^2-1}(q+1),
\end{align*}
\begin{align*}
    \ABC(\mathcal{B}(G))&=\sqrt{p^2-1}\bigg(\sqrt{p^2-2}+p\sqrt{p^4-p^2-1}+\sqrt{(q^2-1)(p^2q^2-p^2-q^2)} \\
    & \qquad \qquad \qquad \qquad \qquad +q\sqrt{(q^2-1)(p^2q^4-p^2q^2-q^4+q^2-1)} \\
    & \qquad \qquad \qquad \qquad \qquad +p\sqrt{(q^2-1)(p^4q^2-p^2q^2-p^4+p^2-1}) \\
    & \qquad \qquad \qquad \qquad \qquad +pq\sqrt{(q^2-1)(p^4q^4-p^2q^4-p^4q^2+p^2q^2-1})\bigg) \\
    & \qquad \qquad \qquad \qquad \qquad +\sqrt{(q^2-1)}\big(\sqrt{q^2-2}+q\sqrt{q^4-q^2-1}\big),
\end{align*}
\begin{align*}
    \GA(\mathcal{B}(G))& =1+2\sqrt{(p^2-1)^3}\bigg(\frac{1}{p^2}+\frac{p^3}{p^4-p^2+1}+\frac{\sqrt{(q^2-1)^3}}{p^2q^2-p^2-q^2+2} \\
    & \qquad +\frac{q^3\sqrt{(q^2-1)^3}}{p^2q^4-p^2q^2-q^4+q^2+1}+\frac{p^3\sqrt{(q^2-1)^3}}{p^4q^2-p^2q^2-p^4+p^2+1} \\
    & \qquad +\frac{p^3q^3\sqrt{(q^2-1)^3}}{p^4q^4-p^2q^4-p^4q^2+p^2q^2+1}\bigg)+2\sqrt{(q^2-1)^3}\bigg(\frac{1}{q^2}+\frac{q^3}{q^4-q^2+1}\bigg),
\end{align*}
\begin{align*}
    H(\mathcal{B}(G))&=1+2(p^2-1)\bigg(\frac{1}{p^2}+\frac{p^2}{p^4-p^2+1}+\frac{q^2-1}{p^2q^2-p^2-q^2+2} \\
    & \qquad +\frac{p^2(q^2-1)}{p^4q^2-p^2q^2-p^4+p^2+1}+\frac{q^2(q^2-1)}{p^2q^4-p^2q^2-q^4+q^2+1} \\
    & \qquad +\frac{p^2q^2(q^2-1)}{p^4q^4-p^2q^4-p^4q^2+p^2q^2+1}\bigg)+2(q^2-1)\bigg(\frac{1}{q^2}+\frac{q^2}{q^4-q^2+1}\bigg) \text{ and }
\end{align*}	  
\begin{align*}
    \SCI(\mathcal{B}(G))&=\frac{1}{\sqrt{2}}+(p^2-1)\bigg(\frac{1}{p}+\frac{p^2}{\sqrt{p^4-p^2+1}}+\frac{q^2-1}{\sqrt{p^2q^2-p^2-q^2+2}} \\
    & \qquad +\frac{p^2(q^2-1)}{p^4q^2-p^2q^2-p^4+p^2+2}+\frac{q^2(q^2-1)}{\sqrt{p^2q^4-p^2q^2-q^4+q^2+1}} \\
    & \qquad +\frac{p^2q^2(q^2-1)}{\sqrt{p^4q^4-p^2q^4-p^4q^2+p^2q^2+1}}\bigg)+(q^2-1)\bigg(\frac{1}{q}+\frac{q^2-1}{\sqrt{q^4-q^2+1}}\bigg).
\end{align*}
\end{theorem}

\section{Various spectra and energies of SGB-graph}
    Let $\mathcal{G}$ be a simple graph with vertex set $V(\mathcal{G})=\{v_i:i=1,2, \ldots, n\}$ and $|e(\mathcal{G})|=m$. Let $A(\mathcal{G})$ and $D(\mathcal{G})$ denote the adjacency matrix and degree matrix of $\mathcal{G}$ respectively. The set of eigenvalues of $A(\mathcal{G})$ along with their multiplicities is defined as the spectrum of $\mathcal{G}$. The Laplacian matrix and signless Laplacian matrix of $\mathcal{G}$ are given by $L(\mathcal{G}):=D(\mathcal{G})-A(\mathcal{G})$ and $Q(\mathcal{G}):=D(\mathcal{G})+A(\mathcal{G})$ respectively. The Laplacian spectrum (L-spectrum) and signless Laplacian spectrum (Q-spectrum) of $\mathcal{G}$ is the set of eigenvalues of $L(\mathcal{G})$ and $Q(\mathcal{G})$ along with their multiplicities respectively. The common neighborhood of two distinct vertices $v_i$ and $v_j$, denoted by $C(v_i, v_j)$, is the set of all vertices other than $v_i$ and $v_j$ which are adjacent to both $v_i$ and $v_j$. The common neighborhood matrix of $\mathcal{G}$, denoted by $CN(\mathcal{G})$, is defined as
\[
(CN(\mathcal{G}))_{i, j}= \begin{cases}
	|C(v_i, v_j)|, & \text{if } i \neq j \\
	0, & \text{if } i=j.
\end{cases}
\]
The set of all eigenvalues of $CN(\mathcal{G})$ along with their multiplicities is called the common neighborhood spectrum (CN-spectrum) of $\mathcal{G}$. We write $\Spec(\mathcal{G})$, $\L-Spec(\mathcal{G})$, $\Q-Spec(\mathcal{G})$ and $\CN-Spec(\mathcal{G})$ to denote the spectrum, L-spectrum, Q-spectrum and CN-spectrum of $\mathcal{G}$ respectively. These multi-sets $\Spec(\mathcal{G})$/$\L-Spec(\mathcal{G})$/$\Q-Spec(\mathcal{G})$/$\CN-Spec(\mathcal{G})$ are described as  $\{(x_1)^{n_1}, (x_2)^{n_2}, \dots, (x_k)^{n_k}\}$ where $x_i$'s are eigenvalues of $A(\mathcal{G})$/$L(\mathcal{G})$/$Q(\mathcal{G})$/$CN(\mathcal{G})$ and $n_i$'s are their multiplicities. The graph $\mathcal{G}$ is called integral/L-integral/Q-integral/CN-integral if  $\Spec(\mathcal{G})$/$\L-Spec(\mathcal{G})$/$\Q-Spec(\mathcal{G})$/$\CN-Spec(\mathcal{G})$ contain only integers. %In graph theory, it is a general problem to determine integral, L-integral, Q-integral and CN-integral graphs.
%(see \cite{HaSc-1974,BaCv-2002,FKMN-2005,Merris-1994,CRS-2007,Stanic-2009,ASG}).

The energy, $E(\mathcal{G})$ and common neighborhood energy (CN-energy), $E_{CN}(\mathcal{G})$ of $\mathcal{G}$ are the sum of the absolute values of the eigenvalues of $A(\mathcal{G})$ and $CN(\mathcal{G})$ respectively. Thus 
\[
E(\mathcal{G}) = \sum_{\alpha \in \Spec(\mathcal{G})}|\alpha| \quad \text{ and } \quad E_{CN}(\mathcal{G}) = \sum_{\beta \in \CN-Spec(\mathcal{G})}|\beta|.
\]
 The Laplacian energy (L-energy), $LE(\mathcal{G})$ and signless Laplacian energy (Q-energy), $LE^{+}(\mathcal{G})$ of $\mathcal{G}$ are defined as
\[
LE(\mathcal{G})=\sum_{\lambda \in \L-Spec(\mathcal{G})} \left| \lambda-\frac{2m}{n} \right| \quad \text{and} \quad LE^{+}(\mathcal{G})= \sum_{\mu \in \Q-Spec(\mathcal{G})} \left| \mu-\frac{2m}{n} \right|.
\]
Note that energy of a graph was introduced by Gutman \cite{Gutman-78} and   the other graph energies mentioned above were introduced by Gutman et al. \cite{GZ06,ACGMR11,ASG}.
It is well-known that $E(K_n)=LE(K_n)=LE^+(K_n)=2(n-1)$ and $E_{CN}(K_n)=2(n-1)(n-2)$.
A graph $\mathcal{G}$ with $|V(\mathcal{G})| =n$  is called hyperenergetic if $E(\mathcal{G}) > E(K_n)$. It is called hypoenergetic if $E(\mathcal{G}) < n$. Similarly, $\mathcal{G}$ is called L-hyperenergetic if $LE(\mathcal{G}) > LE(K_n)$, Q-hyperenergetic if $LE^+(\mathcal{G}) > LE^+(K_n)$ and CN-hyperenergetic if $E_{CN}(\mathcal{G}) > E_{CN}(K_n)$. 
It is still an open problem to find a CN-hyperenergetic graph (see \cite[Open Problem 1]{ASG}).

The common neighborhood graph of $\mathcal{G}$, denoted by  $\con(\mathcal{G})$,  is a graph whose adjacency matrix is given by
\begin{align*}
	(A(\con(\mathcal{G})))_{i,j}=\begin{cases}
		1, & \text{if } |C(v_i, v_j)| \geq 1 \text{ and } i \neq j \\
		0, & \text{otherwise.}
	\end{cases}
\end{align*}
It is easy to see that $CN(\mathcal{G}) = A(\con(\mathcal{G}))$  if   $\mathcal{G}=K_{1, n}$, the star on $n+1$ vertices. In  \cite[Example 2.1]{AAGS12}, it was shown that $\con(K_{1, n}) =  K_1 \sqcup K_n$. Therefore,  $\CN-Spec(K_{1, n})= \Spec(K_1) \sqcup \Spec(K_n)$.

 In this section, we compute spectrum, L-spectrum,  Q-spectrum, CN-spectrum and their corresponding energies of $\mathcal{B}(G)$ for the groups considered in Section 2 including cyclic groups of order $p^n$ for any prime $p$ and $n \geq 1$. Consequently, we shall show that $\mathcal{B}(G)$ is not integral but L-integral, Q-integral and CN-integral for these groups. Further, we shall show that $\mathcal{B}(G)$ is hypoenergetic but neither hyperenergetic, L-hyperenergetic, Q-hyperenergetic nor CN-hyperenergetic if $G$ is one of the above mentioned groups. 
%Note that it is still an open problem to find a CN-hyperenergetic graph (see \cite[Open Problem 1]{ASG}).  
Gutman et al. \cite{E-LE-Gutman}  conjectured that $E(\mathcal{G}) \leq LE(\mathcal{G})$  which is known as E-LE conjecture. Gutman \cite{Gutman-78} also conjectured that ``$\mathcal{G}$ is not hyperenergetic if $\mathcal{G} \ncong K_{|v(\mathcal{G})|}$". However, both the conjectures were disproved (see \cite{Liu-09,SSM,Gutman-2011,DNN-24,SN24}). We shall show that $\mathcal{B}(G)$ satisfies both the above mentioned conjectures when $G$ is one of the above mentioned groups. 
%We conclude the paper showing that  $\mathcal{B}(G)$ satisfies E-LE conjecture for these groups.

The following well-known results are  useful in computing various spectra and energies of $\mathcal{B}(G)$.
\begin{lemma}\label{all_spec_of_star}
	If \,$\mathcal{G}=K_{1, n}$, the star on $n+1$ vertices, then
	$\Spec(\mathcal{G})=\big{\{}(0)^{n-1}, \left(\pm \sqrt{n}\right)^1\big{\}}$, 
	$\L-Spec(\mathcal{G})=\left\lbrace(0)^1, (1)^{n-1}, (n+1)^1\right\rbrace \, = \,  \Q-Spec(\mathcal{G})$ and
	%$\Q-Spec(\mathcal{G})=\left\lbrace(0)^1, (1)^{n-1}, (n+1)^1\right\rbrace$ and
	$\CN-Spec(\mathcal{G}) = \big{\{}(0)^1, (-1)^{n-1},$ $ (n-1)^1\big{\}}$.
\end{lemma}
\begin{lemma}\label{all_energy_of_star}
	If \,$\mathcal{G}=K_{1, n}$, the star on $n+1$ vertices, then
	$E(\mathcal{G})=2 \sqrt{n}$, 
	$LE(\mathcal{G})=\frac{2n^2+2}{n+1} = LE^+(\mathcal{G})$
	%$LE^+(\mathcal{G})=\frac{2n^2+2}{n+1}$ 
	and 
	$E_{CN}(\mathcal{G})= 2n-2$. 
\end{lemma}
\begin{remark}
Since $\mathcal{B}(G)$ is the union of some stars, in view of Lemma \ref{all_spec_of_star}, we have $\L-Spec(\mathcal{B}(G)) = \Q-Spec(\mathcal{B}(G))$ and so $LE(\mathcal{B}(G))= LE^+(\mathcal{B}(G))$.
\end{remark}

\begin{theorem}\label{all_spec_cyclic_p^n}
    If $G$ is a cyclic group of order $p^n$, where $p$ is any prime and $n \geq 1$, then
%    \begin{align*}
%		\Spec(\mathcal{B}(G))&=\left\lbrace(0)^{p^{2n}-n-1}, (-1)^1, (1)^1, \left(\sqrt{p^2-1}\right)^1, \left(-\sqrt{p^2-1}\right)^1, \left(p\sqrt{p^2-1}\right)^1, \right. \\
%		&\qquad \qquad  \left. \left(-p\sqrt{p^2-1}\right)^1, \left(p^2\sqrt{p^2-1}\right)^1, \left(-p^2\sqrt{p^2-1}\right)^1, \ldots, \right. \\
%        & \qquad \qquad \qquad \qquad \qquad \qquad \left. \left(p^{n-1}\sqrt{p^2-1}\right)^1, \left(-p^{n-1}\sqrt{p^2-1}\right)^1 \right\rbrace,
%	\end{align*} 
\begin{center}
	 {\small{$\Spec(\mathcal{B}(G))=\left\lbrace(0)^{p^{2n}-n-1}, (\pm1 )^1,  \left(\pm \sqrt{p^2-1}\right)^1,  \left(\pm p\sqrt{p^2-1}\right)^1, \dots,  \left(\pm p^{n-1}\sqrt{p^2-1}\right)^1 \right\rbrace$,}}     
\end{center}	
\begin{align*}
	\L-Spec(\mathcal{B}(G))&=\left\lbrace(0)^{n+1}, (1)^{p^{2n}-n-1}, (2)^1, (p^2)^1, (p^4-p^2+1)^1, (p^6-p^4+1)^1, \ldots, \right.\\
    & \qquad \qquad \qquad \qquad \qquad \quad \left. (p^{2n}-p^{2n-2}+1)\right\rbrace = \Q-Spec(\mathcal{B}(G))
	\end{align*}
and $\CN-Spec(\mathcal{B}(G))=\lbrace(0)^{n+2}, (-1)^{p^{2n}-n-1},(p^2-2)^1, (p^4-p^2-1)^1, (p^6-p^4-1)^1, \ldots$, $(p^{2n}-p^{2n-2}-1)^1\rbrace$.
\end{theorem}
\begin{proof}
    From Theorem \ref{structure_cyclic_p^n}, we have $\mathcal{B}(G)=K_2 \sqcup K_{1, p^2-1} \sqcup K_{1, p^2(p^2-1)} \sqcup \cdots \sqcup K_{1, p^{2n-2}(p^2-1)}$. Now, by Lemma \ref{all_spec_of_star}, we get 
\begin{center}
    $\Spec(K_2)=\{(\pm 1)^1\}, \, \Spec(K_{1, p^2-1})=\left\{(0)^{p^2-2}, (\pm \sqrt{p^2-1})^1\right\},$ 
    $\Spec(K_{1, p^2(p^2-1)})=\left\{(0)^{p^4-p^2-1}, (\pm p\sqrt{p^2-1})^1\right\}$, $\ldots$, $\Spec(K_{1, p^{2n-2}(p^2-1)})=\left\{(0)^{p^{2n}-p^{2n-2}-1}, (\pm p^{n-1}\sqrt{p^2-1})^1\right\}$. 
  \end{center}   
Since we have \,  $\Spec(\mathcal{B}(G))\, = \, \Spec(K_2) \, \sqcup \, \Spec(K_{1, p^2-1}) \, \sqcup \, \Spec(K_{1, p^2(p^2-1)}) \, \sqcup \, \cdots \, \sqcup \, $ $ \Spec(K_{1, p^{2n-2}(p^2-1)})$,  we get the desired expression of $\Spec(\mathcal{B}(G))$. Also, by Lemma \ref{all_spec_of_star}, we get
\begin{center}
    $\L-Spec(K_2)=\{(0)^1, (2)^1\},$ $ \L-Spec(K_{1, p^2-1})=\left\{ (0)^1, (1)^{p^2-2}, (p^2)^1\right\},$ $\L-Spec(K_{1, p^2(p^2-1)})=\left\{ (0)^1, (1)^{p^4-p^2-1}, (p^4-p^2+1)^1\right\} $, $\ldots$,
   $\L-Spec(K_{1, p^{2n-2}(p^2-1)})=\left\{(0)^1, (1)^{p^{2n}-p^{2n-2}-1}, (p^{2n}-p^{2n-2}+1)^1\right\}$. 
\end{center}
Since \quad $\L-Spec(\mathcal{B}(G)) =\L-Spec(K_2) \, \sqcup \, \L-Spec(K_{1, p^2-1}) \, \sqcup \, \L-Spec(K_{1, p^2(p^2-1)}) \, \sqcup \, \ldots \, \sqcup$ $\L-Spec(K_{1, p^{2n-2}(p^2-1)})$
%\begin{align*}
%    & \quad \L-Spec(\mathcal{B}(G)) \\
%    & =\L-Spec(K_2) \sqcup \L-Spec(K_{1, p^2-1}) \sqcup \L-Spec(K_{1, p^2(p^2-1)}) \sqcup \ldots \sqcup \L-Spec(K_{1, p^{2n-2}(p^2-1)}) 
%    &  =\Q-Spec(K_2) \sqcup \Q-Spec(K_{1, p^2-1}) \sqcup \Q-Spec(K_{1, p^2(p^2-1)}) \sqcup \ldots \sqcup \Q-Spec(K_{1, p^{2n-2}(p^2-1)}) \\
%    &  = \Q-Spec(\mathcal{B}(G))
%\end{align*}
we get the desired expression for $\L-Spec(\mathcal{B}(G))$ and hence the expression for $\Q-Spec(\mathcal{B}(G))$. Further, we have
\begin{center}
    $\CN-Spec(K_2)=\{(0)^2\}, \, \CN-Spec(K_{1, p^2-1})=\left\{(0)^1, (-1)^{p^2-2}, (p^2-2)^1\right\},$ 
    $\CN-Spec(K_{1, p^2(p^2-1)})=\left\{(0)^1, (-1)^{p^4-p^2-1}, (p^4-p^2-1)^1\right\}$, $\ldots$, 
    $\CN-Spec(K_{1, p^{2n-2}(p^2-1)})=\left\{(0)^1, (-1)^{p^{2n}-p^{2n-2}-1}, (p^{2n}-p^{2n-2}-1)^1\right\}$. 
\end{center}
Since  

\noindent $\CN-Spec(\mathcal{B}(G))$

\qquad $=\CN-Spec(K_2) \, \sqcup \, \CN-Spec(K_{1, p^2-1}) \, \sqcup \, \CN-Spec(K_{1, p^2(p^2-1)}) \sqcup \cdots \, \sqcup $ $\CN-Spec(K_{1, p^{2n-2}(p^2-1)})$,  

\noindent we get the desired expression of $\CN-Spec(\mathcal{B}(G))$.
\end{proof}
\begin{theorem}\label{all_energy_cyclic_p^n}
	If $G$ is a cyclic group of order $p^n$, where $p$ is any prime and $n \geq 1$, then
	$E(\mathcal{B}(G))= 2+2\sqrt{p^2-1}\left(\frac{p^n-1}{p-1}\right)$,
	$LE(\mathcal{B}(G))= LE^+(\mathcal{B}(G))=\frac{2p^{4n}+2n^2+4n+2}{p^{2n}+n+1}$ 
	\quad  and
	$E_{CN}(\mathcal{B}(G))=2p^{2n}-2n-2$.
\end{theorem}
\begin{proof}
	From Theorem \ref{all_spec_cyclic_p^n} and definition of $E(\mathcal{B}(G))$ and $E_{CN}(\mathcal{B}(G))$, we have
	\begin{align*}
		E(\mathcal{B}(G))&=2+2\sqrt{p^2-1}+2p\sqrt{p^2-1}+2p^2\sqrt{p^2-1}+\cdots+2p^{n-1}\sqrt{p^2-1} \\
		&= 2+2\sqrt{p^2-1}\left(1+p+p^2+\cdots+p^{n-1}\right) 
		%&=2+2\sqrt{p^2-1}\left(\frac{p^n-1}{p-1}\right) \qquad 
	\end{align*}
and $E_{CN}(\mathcal{B}(G))=(\underbrace{1+1+\cdots+1}
_{(p^{2n}-n-1)\text{-times}})+(p^2-2)+(p^4-p^2-1) +(p^6-p^4-1)+ \cdots$  $+(p^{2n}-p^{2n-2}-1)$.
%	\begin{align*}
%\text{ and }		E_{CN}(\mathcal{B}(G))=(\underbrace{1+1+\ldots+1}
%		_{(p^{2n}-n-1)\text{-times}})&+(p^2-2)+(p^4-p^2-1)\\
%		& +(p^6-p^4-1)+ \cdots  +(p^{2n}-p^{2n-2}-1). 
%	%	&= 2p^{2n}-2n-2.
%	\end{align*}
Thus we get the required expressions for $E(\mathcal{B}(G))$ and $E_{CN}(\mathcal{B}(G))$.

	Further, from Theorem \ref{all_spec_cyclic_p^n} and Theorem \ref{structure_cyclic_p^n}, we have 
	
\noindent	$\L-Spec(\mathcal{B}(G))=  \Q-Spec(\mathcal{B}(G))$
	
	\qquad\qquad\quad$=\lbrace(0)^{n+1}, (1)^{p^{2n}-n-1}$, $(2)^1, (p^2)^1, (p^4-p^2+1)^1, (p^6-p^4+1)^1, \ldots, (p^{2n}-p^{2n-2}+1)\rbrace$ 
	
	\noindent and $\frac{2m}{n}=\frac{2p^{2n}}{p^{2n}+n+1}$. Now, $|0-\frac{2p^{2n}}{p^{2n}+n+1}|=\frac{2p^{2n}}{p^{2n}+n+1}$, $|1-\frac{2p^{2n}}{p^{2n}+n+1}|=\frac{p^{2n}-n-1}{p^{2n}+n+1}$ and $|2-\frac{2p^{2n}}{p^{2n}+n+1}|=\frac{2n+2}{p^{2n}+n+1}$. For all $ 1\leq x\leq n$, we have 
	\begin{align*}
		 (p^{2x}-p^{2x-2}&+1)-\frac{2p^{2n}}{p^{2n}+n+1} \\
		&=\frac{p^{2n}(p^{2x}-p^{2x-2}-1)+np^{2x-2}(p^2-1)+p^{2x-2}(p^2-1)+n+1}{p^{2n}+n+1}>0, 
	\end{align*}
	since $p^{2x-2}(p^2-1)>1$ and $p^2-1>0$ for all $p$. As such, $|(p^{2x}-p^{2x-2}+1)-\frac{2p^{2n}}{p^{2n}+n+1}|=(p^{2x}-p^{2x-2}+1)-\frac{2p^{2n}}{p^{2n}+n+1}$ for all $1\leq x\leq n$. Therefore, by definition of (signless) Laplacian energy, we have
	\begin{align*}
		LE(\mathcal{B}(G))&= LE^+(\mathcal{B}(G))\\
		& = (n+1) \frac{2p^{2n}}{p^{2n}+n+1}+(p^{2n}-n-1)\frac{p^{2n}-n-1}{p^{2n}+n+1}+\frac{2n+2}{p^{2n}+n+1} \\
		& \qquad +\left((p^2-1+1)-\frac{2p^{2n}}{p^{2n}+n+1}\right)+\left((p^4-p^2+1)-\frac{2p^{2n}}{p^{2n}+n+1}\right) \\
		& \qquad   +\cdots +\left((p^{2n}-p^{2n-2}+1)-\frac{2p^{2n}}{p^{2n}+n+1}\right) \\
		&=(n+1)\frac{2p^{2n}}{p^{2n}+n+1}+(p^{2n}-n-1)\frac{p^{2n}-n-1}{p^{2n}+n+1}+\frac{2n+2}{p^{2n}+n+1} \\
		& \quad \quad \quad \quad \quad \quad \quad \quad  \quad\,\, +\left(p^{2n}+n-1-n \frac{2p^{2n}}{p^{2n}+n+1}\right).
		%&=\frac{2p^{4n}+2n^2+4n+2}{p^{2n}+n+1}.
	\end{align*}
Hence, we get the required expression.
%	This completes the proof.
\end{proof}
\begin{theorem}\label{all_energy_comparison_cyclic_p^n}
	If $G$ is a cyclic group of order $p^n$, where $p$ is any prime and $n \geq 1$, then $\mathcal{B}(G)$ is hypoenergetic but not hyperenergetic,  L-hyperenergetic, Q-hyperenergetic and CN-hyperenergetic.
\end{theorem}
\begin{proof}
	By Theorem \ref{structure_cyclic_p^n} and Theorem \ref{all_energy_cyclic_p^n}, we have $|V(\mathcal{B}(G))|=p^{2n}+n+1$ and $E(\mathcal{B}(G))=2+2\sqrt{p^2-1}+2p\sqrt{p^2-1}+2p^2\sqrt{p^2-1}+\cdots+2p^{n-1}\sqrt{p^2-1}$.
	Since $p^2>2\sqrt{p^2-1}$, $p^4-p^2+1>2\sqrt{p^4-p^2}=2p\sqrt{p^2-1}$, $p^6-p^4+1>2\sqrt{p^6-p^4}=2p^2\sqrt{p^2-1}$, $\ldots$, $p^{2n}-p^{2n-2}+1 >2\sqrt{p^{2n}-p^{2n-2}}=2p^{n-1}\sqrt{p^2-1}$ \,\, we have 
	\begin{align*}
		p^2+(p^4-p^2+1)+&(p^6-p^4+1)+\cdots+(p^{2n}-p^{2n-2}+1)+2 \\
		& > 2+2\sqrt{p^2-1}+2p\sqrt{p^2-1}+2p^2\sqrt{p^2-1}+\cdots+2p^{n-1}\sqrt{p^2-1}.
	\end{align*}
	Therefore,
	\begin{equation} \label{cyclic-p^n-hypo}
		|V(\mathcal{B}(G))|> E(\mathcal{B}(G)).
	\end{equation}	 
	Thus, $\mathcal{B}(G)$ is hypoenergetic. 
	
	We have $E(K_{|V(\mathcal{B}(G))|}) = E(K_{p^{2n}+n+1})= 2(p^{2n}+n+1-1)=2p^{2n}+2n >p^{2n}+n+1= |V(\mathcal{B}(G))|> E(\mathcal{B}(G))$ (using \eqref{cyclic-p^n-hypo}). Therefore,  $\mathcal{B}(G)$ is not hyperenergetic.

	Also, $LE(K_{p^{2n}+n+1})=LE^+(K_{p^{2n}+n+1})=2p^{2n}+2n$. From Theorem \ref{all_energy_cyclic_p^n}, we have $LE(\mathcal{B}(G))$ $=LE^+(\mathcal{B}(G))=\frac{2p^{4n}+2n^2+4n+2}{p^{2n}+n+1}$. Now, 
	\begin{align*}
		LE(K_{p^{2n}+n+1})-LE(\mathcal{B}(G))&= LE^+(K_{p^{2n}+n+1})-LE^+(\mathcal{B}(G)) \\
		&=2p^{2n}+2n-\frac{2p^{4n}+2n^2+4n+2}{p^{2n}+n+1} \\
		&= \frac{4np^{2n}+2p^{2n}-2n-2}{p^{2n}+n+1} > 0.
	\end{align*}
	Therefore, 
	\[
	LE^+(K_{p^{2n}+n+1}) = LE(K_{p^{2n}+n+1}) > LE(\mathcal{B}(G))=LE^+(\mathcal{B}(G)).
	\] 
	Hence,  $\mathcal{B}(G)$ is neither L-hyperenergetic nor  Q-hyperenergetic.
	
	We have  $E_{CN}(K_{p^{2n}+n+1})=2(p^{2n}+n+1-1)(p^{2n}+n+1-2)=(2p^{2n}+2n)(p^{2n}+n-1) > 2p^{2n}-2n-2 = E_{CN}(\mathcal{B}(G))$ (using Theorem \ref{all_energy_cyclic_p^n}). 
	Hence, $\mathcal{B}(G)$ is not CN-hyperenergetic. This completes the proof.
\end{proof}

\begin{theorem}
	If $G$ is a cyclic group of order $p^n$, where $p$ is any prime and $n \geq 1$, then $E(\mathcal{B}(G))<LE(\mathcal{B}(G))$.
\end{theorem}
\begin{proof}
	By Theorem \ref{all_energy_cyclic_p^n}, we get
	$E(\mathcal{B}(G)) =2+2\sqrt{p^2-1}\left(\frac{p^n-1}{p-1}\right)$ and $LE(\mathcal{B}(G)) =\frac{2p^{4n}+2n^2+4n+2}{p^{2n}+n+1}$. Also, from Theorem \ref{all_energy_comparison_cyclic_p^n}, we have $|V(\mathcal{B}(G))|=p^{2n}+n+1> E(\mathcal{B}(G))$. Now, for all $p \geq 2$ and  $n \geq 1$, we have
	\begin{align*}
		LE(\mathcal{B}(G))-|V(\mathcal{B}(G))| 
		&=\frac{2p^{4n}+2n^2+4n+2}{p^{2n}+n+1}-p^{2n}+n+1 \\
		&= \frac{p^{2n}(p^{2n}-2n-2)+n^2+2n+1}{p^{2n}+n+1} >0.
	\end{align*}
	Hence, $LE(\mathcal{B}(G))>|V(\mathcal{B}(G))|>E(\mathcal{B}(G))$.
\end{proof}

\begin{theorem}\label{all_spec_cyclic_pq}
If $G$ is a cyclic group of order $pq$, where $p$ and $q$ are any distinct primes, then
\[
\Spec(\mathcal{B}(G))=\left\lbrace\!(0)^{p^2q^2-4}, (\pm 1)^1,  \left(\pm \sqrt{p^2-1}\right)^1\!\!,  \left(\pm \sqrt{q^2-1}\right)^1\!, \left(\pm\sqrt{(p^2-1)(q^2-1)}\right)^1\right\rbrace, 
\]        
	\[
	\L-Spec(\mathcal{B}(G))=\lbrace(0)^{4}, (1)^{p^2q^2-4}, (2)^1, (p^2)^1, (q^2)^1, (p^2q^2-p^2-q^2+2)^1\rbrace = \Q-Spec(\mathcal{B}(G))
	\]
and	\, $\CN-Spec(\mathcal{B}(G))=\left\lbrace(0)^{5}, (-1)^{p^2q^2-4},(p^2-2)^1, (q^2-2)^1, (p^2q^2-p^2-q^2)^1\right\rbrace$.
\end{theorem}
\begin{proof}
From Theorem \ref{structure_cyclic_pq}, we have $\mathcal{B}(G)=K_2 \sqcup K_{1, p^2-1} \sqcup K_{1, q^2-1} \sqcup K_{1, p^2q^2-p^2-q^2+1}$. Now, by Lemma \ref{all_spec_of_star}, we get 
\begin{center}
    $\Spec(K_2)=\{(\pm 1)^1\}, \, \Spec(K_{1, p^2-1})=\left\{(0)^{p^2-2}, (\pm \sqrt{p^2-1})^1\right\},$ 
    $\Spec(K_{1, q^2-1})=\left\{(0)^{q^2-2}, (\pm\sqrt{q^2-1})^1\right\}$
\end{center}   
    and $\Spec(K_{1, p^2q^2-p^2-q^2+1})=\left\{(0)^{p^2q^2-p^2-q^2}, (\pm\sqrt{(p^2-1)(q^2-1)})^1\right\}$.

Since $\Spec(\mathcal{B}(G))=\Spec(K_2) \sqcup \Spec(K_{1, p^2-1}) \sqcup \Spec(K_{1, q^2-1}) \sqcup \Spec(K_{1, p^2q^2-p^2-q^2+1})$,  we get the desired expression of $\Spec(\mathcal{B}(G))$. Also, by Lemma \ref{all_spec_of_star}, we get
\begin{center}
    $\L-Spec(K_2)=\{(0)^1, (2)^1\},$ $ \L-Spec(K_{1, p^2-1})=\left\{ (0)^1, (1)^{p^2-2}, (p^2)^1\right\},$ $\L-Spec(K_{1, q^2-1})=\left\{ (0)^1, (1)^{q^2-2}, (q^2)^1\right\} $ 
\end{center}     
    and   
$\L-Spec(K_{1, p^2q^2-p^2-q^2+1})=\left\{(0)^1, (1)^{p^2q^2-p^2-q^2}, (p^2q^2-p^2-q^2+2)^1\right\}$. 
Since 

\noindent $\L-Spec(\mathcal{B}(G))\!=\!\L-Spec(K_2) \sqcup \L-Spec(K_{1, p^2-1}) \sqcup \L-Spec(K_{1, q^2-1}) \sqcup \L-Spec(K_{1, p^2q^2-p^2-q^2+1})$
\noindent we get the desired expression for $\L-Spec(\mathcal{B}(G))$ and hence the expression for $\Q-Spec(\mathcal{B}(G))$. Further, we have
\begin{center}
    $\CN-Spec(K_2)=\{(0)^2\}, \, \CN-Spec(K_{1, p^2-1})=\left\{(0)^1, (-1)^{p^2-2}, (p^2-2)^1\right\},$ 
    $\CN-Spec(K_{1, q^2-1})=\left\{(0)^1, (-1)^{q^2-2}, (q^2-2)^1\right\}$ 
\end{center}    
    and 
    $\CN-Spec(K_{1, p^2q^2-p^2-q^2+1})=\left\{(0)^1, (-1)^{p^2q^2-p^2-q^2}, (p^2q^2-p^2-q^2)^1\right\}$.

Since   

\noindent $\CN-Spec(\mathcal{B}(G))=  \CN-Spec(K_2)  \sqcup\CN-Spec(K_{1, p^2-1})  \sqcup  \CN-Spec(K_{1, q^2-1})$ 
$ \sqcup \CN-Spec(K_{1, p^2q^2-p^2-q^2+1})$, we get the desired expression of $\CN-Spec(\mathcal{B}(G))$.
\end{proof}
\begin{theorem}\label{all_energy_cyclic_pq}
	If $G$ is a cyclic group of order $pq$, where $p$ and $q$ are two  primes such that $p < q$, then
	$E(\mathcal{B}(G))= 2+2\sqrt{p^2-1}(1+\sqrt{q^2-1})+2\sqrt{q^2-1}$,
	$LE(\mathcal{B}(G))= LE^+(\mathcal{B}(G))=\frac{2p^4q^4+32}{p^2q^2+4}$ 
	\quad  and
	$E_{CN}(\mathcal{B}(G))=2p^2q^2-8$.
\end{theorem}
\begin{proof}
	From Theorem \ref{all_spec_cyclic_pq} and definition of $E(\mathcal{B}(G))$ and $E_{CN}(\mathcal{B}(G))$, we have
	\begin{align*}
		E(\mathcal{B}(G))&=2+2\sqrt{p^2-1}+2\sqrt{q^2-1}+2\sqrt{(p^2-1)(q^2-1)}
%	&= 2+2\sqrt{p^2-1}(1+\sqrt{q^2-1})+2\sqrt{q^2-1} \qquad \text{ and }
	\end{align*}
and
		$E_{CN}(\mathcal{B}(G))=(\underbrace{1+1+\cdots+1}
		_{(p^2q^2-4)\text{-times}})+p^2-2+q^2-2+p^2q^2-p^2-q^2$. 
%&= 2p^2q^2-8.
Thus we get the required expressions for $E(\mathcal{B}(G))$ and $E_{CN}(\mathcal{B}(G))$.
	
	Further, from Theorem \ref{all_spec_cyclic_pq} and Theorem \ref{structure_cyclic_pq}, we have 
\begin{center}	
	$\L-Spec(\mathcal{B}(G))= \Q-Spec(\mathcal{B}(G))=\lbrace(0)^{4}, (1)^{p^2q^2-4}, (2)^1, (p^2)^1, (q^2)^1, (p^2q^2-p^2-q^2+2)^1\rbrace$ 
\end{center}	
	\noindent and $\frac{2m}{n}=\frac{2p^2q^2}{p^2q^2+4}$. Now, $|0-\frac{2p^2q^2}{p^2q^2+4}|=\frac{2p^2q^2}{p^2q^2+4}$, $|1-\frac{2p^2q^2}{p^2q^2+4}|=\frac{p^2q^2-4}{p^2q^2+4}$, $|2-\frac{2p^2q^2}{p^2q^2+4}|=\frac{8}{p^2q^2+4}$, $|p^2-\frac{2p^2q^2}{p^2q^2+4}|=\frac{p^4q^2+4p^2-2p^2q^2}{p^2q^2+4}$, $|q^2-\frac{2p^2q^2}{p^2q^2+4}|=\frac{p^2q^4+4q^2-2p^2q^2}{p^2q^2+4}$ and $|(p^2q^2-p^2-q^2+2)-\frac{2p^2q^2}{p^2q^2+4}|=\frac{p^4q^4-p^4q^2-p^2q^4+4p^2q^2-4p^2-4q^2+8}{p^2q^2+4}$. Therefore, by definition of (signless) Laplacian energy, we have
	\begin{align*}
		LE(\mathcal{B}(G))= LE^+(\mathcal{B}(G))&= 4 \times \frac{2p^2q^2}{p^2q^2+4}+(p^2q^2-4)\frac{p^2q^2-4}{p^2q^2+4}+\frac{8}{p^2q^2+4} \\
		& \qquad +\frac{p^4q^2+4p^2-2p^2q^2}{p^2q^2+4}+\frac{p^2q^4+4q^2-2p^2q^2}{p^2q^2+4} \\
		& \qquad +\frac{p^4q^4-p^4q^2-p^2q^4+4p^2q^2-4p^2-4q^2+8}{p^2q^2+4}.
%		&=\frac{2p^4q^4+32}{p^2q^2+4}.
	\end{align*}
Hence, we get the required expression.
%	This completes the proof.
\end{proof}

\begin{theorem}\label{all_energy_comparison_cyclic_pq}
	If $G$ is a cyclic group of order $pq$, where $p$ and $q$ are twp primes such that $p < q$, then $\mathcal{B}(G)$ is hypoenergetic but not hyperenergetic,  L-hyperenergetic, Q-hyperenergetic and CN-hyperenergetic.
\end{theorem}
\begin{proof}
	By Theorem \ref{structure_cyclic_pq} and Theorem \ref{all_energy_cyclic_pq}, we have 
	\[
	|V(\mathcal{B}(G))|=p^2q^2+4 \quad \text{ and} \quad E(\mathcal{B}(G))=2+2\sqrt{p^2-1}+2\sqrt{q^2-1}+2\sqrt{(p^2-1)(q^2-1)}.
	\]
	Since $p^2>2\sqrt{p^2-1}, q^2>2\sqrt{q^2-1}$ and $p^2q^2-p^2-q^2+2 >2\sqrt{p^2q^2-p^2-q^2+1}=2\sqrt{(p^2-1)(q^2-1)}$ \,\, we have 
	\[
	p^2+q^2+p^2q^2-p^2-q^2+2+2 > 2+2\sqrt{p^2-1}+2\sqrt{q^2-1}+2\sqrt{(p^2-1)(q^2-1)}.
	\]
	Therefore,
	\begin{equation} \label{cyclic-pq-hypo}
		|V(\mathcal{B}(G))|> E(\mathcal{B}(G)).
	\end{equation}	 
	Thus, $\mathcal{B}(G)$ is hypoenergetic. 
	
	We have $E(K_{|V(\mathcal{B}(G))|}) = E(K_{p^2q^2+4})= 2(p^2q^2+4-1)=2p^2q^2+6 >p^2q^2+4= |V(\mathcal{B}(G))|> E(\mathcal{B}(G))$ (using \eqref{cyclic-pq-hypo}). Therefore,  $\mathcal{B}(G)$ is not hyperenergetic.

	Also,     $LE(K_{p^2q^2+4})=LE^+(K_{p^2q^2+4})=2p^2q^2+6$. From Theorem \ref{all_energy_cyclic_pq}, we have $LE(\mathcal{B}(G))$ $=LE^+(\mathcal{B}(G))=\frac{2p^4q^4+32}{p^2q^2+4}$. Now, 
	\begin{align*}
		LE(K_{p^2q^2+4})-LE(\mathcal{B}(G))&= LE^+(K_{p^2q^2+4})-LE^+(\mathcal{B}(G))\\ &=2p^2q^2+6-\frac{2p^4q^4+32}{p^2q^2+4} 
		= \frac{14p^2q^2-8}{p^2q^2+4} > 0.
	\end{align*}
	Therefore, 
	$
	LE^+(K_{p^2q^2+4}) = LE(K_{p^2q^2+4}) > LE(\mathcal{B}(G))=LE^+(\mathcal{B}(G)).
	$
	Hence,  $\mathcal{B}(G)$ is neither L-hyperenergetic nor  Q-hyperenergetic.

	We have  $E_{CN}(K_{p^2q^2+4})=2(p^2q^2+4-1)(p^2q^2+4-2)=(2p^2q^2+6)(p^2q^2+2) > 2p^2q^2-8 = E_{CN}(\mathcal{B}(G))$ (using Theorem \ref{all_energy_cyclic_pq}). 
	Hence, $\mathcal{B}(G)$ is not CN-hyperenergetic. This completes the proof.
\end{proof}
\begin{theorem}
	If $G$ is a cyclic group of order $pq$, where $p$ and $q$ are two primes such that $p <q$, then $E(\mathcal{B}(G))<LE(\mathcal{B}(G))$.
\end{theorem}
\begin{proof}
	By Theorem \ref{all_energy_cyclic_pq}, we get
	$E(\mathcal{B}(G)) =2+2\sqrt{p^2-1}(1+\sqrt{q^2-1})+2\sqrt{q^2-1}$ and	$LE(\mathcal{B}(G))  =\frac{2p^4q^4+32}{p^2q^2+4}$. Also, from Theorem \ref{all_energy_comparison_cyclic_pq}, we have $|V(\mathcal{B}(G))|=p^2q^2+4> E(\mathcal{B}(G))$. Now, for all $p, q$, we have
	\begin{align*}
		LE(\mathcal{B}(G))-|V(\mathcal{B}(G))|&=\frac{2p^4q^4+32}{p^2q^2+4}-p^2q^2+4 \\
		&= \frac{p^4q^4-8p^2q^2+16}{p^2q^2+4} 
		= \frac{p^2q^2(p^2q^2-8)+16}{p^2q^2+4} >0.
	\end{align*}
	Hence, $LE(\mathcal{B}(G))>|V(\mathcal{B}(G))|>E(\mathcal{B}(G))$.
\end{proof}

\begin{theorem}\label{all_spec_cyclic_p^2q}
If $G$ is a cyclic group of order $p^2q$, where $p$ and $q$ are two distinct primes, then
    \begin{align*}
		\Spec(\mathcal{B}(G))=&\left\lbrace(0)^{p^4q^2-6}, (\pm 1)^1,  \left(\pm\sqrt{p^2-1}\right)^1,  \left(\pm\sqrt{q^2-1}\right)^1, \left(\pm p\sqrt{p^2-1}\right)^1,\right. \\
		& \qquad\qquad\qquad\qquad\left. \left(\pm\sqrt{(p^2-1)(q^2-1)}\right)^1, \left(\pm p\sqrt{(p^2-1)(q^2-1)}\right)^1 \right\rbrace, 
	\end{align*}        
	\begin{align*}
	\L-Spec(\mathcal{B}(G))&=\left\lbrace(0)^{6}, (1)^{p^4q^2-6}, (2)^1, (p^2)^1, (q^2)^1, (p^4-p^2+1)^1, (p^2q^2-p^2-q^2+2)^1, \right. \\
    & \qquad \qquad \qquad \qquad \qquad \left. (p^4q^2-p^2q^2-p^4+p^2+1)^1\right\rbrace = \Q-Spec(\mathcal{B}(G))
	\end{align*}
and \quad $\CN-Spec(\mathcal{B}(G)) \quad = \, \lbrace(0)^{7}, \, \,\, (-1)^{p^4q^2-6}, \, \,\, (p^2-2)^1, \, \, \, (q^2-2)^1, \, \,\, (p^4-p^2-1)^1$,  

\noindent $(p^2q^2-p^2-q^2)^1, (p^4q^2-p^2q^2-p^4+p^2-1)^1\rbrace$.
%\begin{align*}
%\CN-Spec(\mathcal{B}(G))&=\lbrace(0)^{7}, (-1)^{p^4q^2-6},(p^2-2)^1, (q^2-2)^1, (p^4-p^2-1)^1, \\
%        & \qquad \qquad (p^2q^2-p^2-q^2)^1, (p^4q^2-p^2q^2-p^4+p^2-1)^1\rbrace.
%\end{align*}
\end{theorem}
\begin{proof}
    From Theorem \ref{structure_cyclic_p^2q}, we have \, $\, \mathcal{B}(G)\, =\, K_2 \, \sqcup \, K_{1, p^2-1} \, \sqcup \, K_{1, p^4-p^2} \, \sqcup \, K_{1, q^2-1} \, \sqcup$ \, $K_{1, p^2q^2-p^2-q^2+1} \, \sqcup \, K_{1, p^4q^2-p^2q^2-p^4+p^2}$. Now, by Lemma \ref{all_spec_of_star}, we get 
\begin{center}
    $\Spec(K_2)=\{(\pm 1)^1\}, \, \Spec(K_{1, p^2-1})=\left\{(0)^{p^2-2}, (\pm\sqrt{p^2-1})^1\right\},$
     $\Spec(K_{1, p^4-p^2})=\left\{(0)^{p^4-p^2-1}, (\pm p\sqrt{p^2-1})^1\right\},$
    $\Spec(K_{1, q^2-1})=\left\{(0)^{q^2-2}, (\pm\sqrt{q^2-1})^1\right\},$ 
    $\Spec(K_{1, p^2q^2-p^2-q^2+1})=\left\{(0)^{p^2q^2-p^2-q^2}, (\pm\sqrt{(p^2-1)(q^2-1)})^1\right\}$ 
\end{center}
    and 
    $\Spec(K_{1, p^4q^2-p^2q^2-p^4+p^2})=\left\{(0)^{p^4q^2-p^2q^2-p^4+p^2-1}, (\pm p\sqrt{(p^2-1)(q^2-1)})^1\right\}$.
   
Since $\Spec(\mathcal{B}(G))$ is the union of the above multi-sets, therefore we get the desired expression of $\Spec(\mathcal{B}(G))$. Also, by Lemma \ref{all_spec_of_star}, we get
\begin{center}
    $\L-Spec(K_2)=\{(0)^1, (2)^1\},$ $ \L-Spec(K_{1, p^2-1})=\left\{ (0)^1, (1)^{p^2-2}, (p^2)^1\right\},$ $ \L-Spec(K_{1, p^4-p^2})=\left\{ (0)^1, (1)^{p^4-p^2-1}, (p^4-p^2+1)^1\right\},$ $\L-Spec(K_{1, q^2-1})=\left\{ (0)^1, (1)^{q^2-2}, (q^2)^1\right\}, $ 
$\L-Spec(K_{1, p^2q^2-p^2-q^2+1})=
    \left\{(0)^1, (1)^{p^2q^2-p^2-q^2}, (p^2q^2-p^2-q^2+2)^1\right\}$
\end{center}    and
$\L-Spec(K_{1, p^4q^2-p^2q^2-p^4+p^2})
    =\left\{\!(0)^1, (1)^{p^4q^2-p^2q^2-p^4+p^2-1}, (p^4q^2-p^2q^2-p^4+p^2+1)^1\right\}$. 
 
Since $\L-Spec(\mathcal{B}(G))$ is the union of the above multi-sets,  we get the desired expression for $\L-Spec(\mathcal{B}(G))$  and hence the expression for $\Q-Spec(\mathcal{B}(G))$. Further, we have
\begin{center}
    $\CN-Spec(K_2)=\{(0)^2\}, \, \CN-Spec(K_{1, p^2-1})=\left\{(0)^1, (-1)^{p^2-2}, (p^2-2)^1\right\},$ $\CN-Spec(K_{1, p^4-p^2})=\left\{(0)^1, (-1)^{p^4-p^2-1}, (p^4-p^2-1)^1\right\},$ 
    $\CN-Spec(K_{1, q^2-1})=\left\{(0)^1, (-1)^{q^2-2}, (q^2-2)^1\right\},$ $\CN-Spec(K_{1, p^2q^2-p^2-q^2+1})=\left\{(0)^1, (-1)^{p^2q^2-p^2-q^2}, (p^2q^2-p^2-q^2)^1\right\}$ 
\end{center}
and $ \CN-Spec(K_{1, p^4q^2-p^2q^2-p^4+p^2}) =$
\begin{align*}
    \left\{(0)^1, (-1)^{p^4q^2-p^2q^2-p^4+p^2-1}, 
     (p^4q^2-p^2q^2-p^4+p^2-1)^1\right\}.
\end{align*} 
Since $\CN-Spec(\mathcal{B}(G))$ is the union of the above multi-sets,  we get the desired expression for $\CN-Spec(\mathcal{B}(G))$.
\end{proof}
\begin{theorem}\label{all_energy_cyclic_p^2q}
	If $G$ is a cyclic group of order $p^2q$, where $p$ and $q$ are two distinct primes, then
	$E(\mathcal{B}(G))= 2+2\sqrt{p^2-1}((1+p)(1+\sqrt{q^2-1}))+2\sqrt{q^2-1}$,
	$LE(\mathcal{B}(G))= LE^+(\mathcal{B}(G))=\frac{2p^8q^4+72}{p^4q^2+6}$ 
	\quad  and \,
	$E_{CN}(\mathcal{B}(G))=2p^4q^2-12$.
\end{theorem}
\begin{proof}
	From Theorem \ref{all_spec_cyclic_p^2q} and definition of $E(\mathcal{B}(G))$ and $E_{CN}(\mathcal{B}(G))$, we have
	\begin{align*}
		E(\mathcal{B}(G))=2+2\sqrt{p^2-1}&+2\sqrt{q^2-1}+2p\sqrt{p^2-1} \\
		&  +2\sqrt{(p^2-1)(q^2-1)} +2p\sqrt{(p^2-1)(q^2-1)}
%		&= 2+2\sqrt{p^2-1}((1+p)(1+\sqrt{q^2-1}))+2\sqrt{q^2-1} \qquad \text{ and }
	\end{align*}
and
$E_{CN}(\mathcal{B}(G))=(\underbrace{1+1+\cdots+1}
		_{(p^4q^2-6)\text{-times}})+p^2-2+q^2-2+p^4-p^2-1  +p^2q^2-p^2-q^2+p^4q^2-p^2q^2-p^4+p^2-1$.
%		&= 2p^4q^2-12.
%
Thus we get the required expressions for $E(\mathcal{B}(G))$ and $E_{CN}(\mathcal{B}(G))$.

	Further, from Theorem \ref{all_spec_cyclic_p^2q} and Theorem \ref{structure_cyclic_p^2q}, we have 
	
\noindent	$\L-Spec(\mathcal{B}(G))= \Q-Spec(\mathcal{B}(G))=\lbrace(0)^{6}, (1)^{p^4q^2-6}, (2)^1, (p^2)^1, (q^2)^1, (p^4-p^2+1)^1, (p^2q^2-p^2-q^2+2)^1, (p^4q^2-p^2q^2-p^4+p^2+1)^1\rbrace $ 
	
\noindent	and $\frac{2m}{n}=\frac{2p^4q^2}{p^4q^2+6}$. Now, $|0-\frac{2p^4q^2}{p^4q^2+6}|=\frac{2p^4q^2}{p^4q^2+6}$, $|1-\frac{2p^4q^2}{p^4q^2+6}|=\frac{p^4q^2-6}{p^4q^2+6}$, $|2-\frac{2p^4q^2}{p^4q^2+6}|=\frac{12}{p^4q^2+6}$, $|p^2-\frac{2p^4q^2}{p^4q^2+6}|=\frac{p^6q^2+6p^2-2p^4q^2}{p^4q^2+6}$, $|q^2-\frac{2p^4q^2}{p^4q^2+6}|=\frac{p^4q^4+6q^2-2p^4q^2}{p^4q^2+6}$, $|(p^4-p^2+1)-\frac{2p^4q^2}{p^4q^2+6}|=\frac{p^8q^2-p^6q^2-p^4q^2+6p^4-6p^2+6}{p^4q^2+6}$, $|(p^2q^2-p^2-q^2+2)-\frac{2p^4q^2}{p^4q^2+6}|=\frac{p^6q^4-p^6q^2-p^4q^4+6p^2q^2-6p^2-6q^2+12}{p^4q^2+6}$ \, and $|(p^4q^2-p^2q^2-p^4+p^2+1)-\frac{2p^4q^2}{p^4q^2+6}|=\frac{p^8q^4-p^6q^4-p^8q^2+p^6q^2+5p^4q^2-6p^2q^2-6p^4+6p^2+6}{p^4q^2+6}$. Therefore, by definition of (signless) Laplacian energy, we have
	\begin{align*}
		LE(\mathcal{B}(G))&= LE^+(\mathcal{B}(G))\\
		& = 6 \times \frac{2p^4q^2}{p^4q^2+6}+(p^4q^2-6)\frac{p^4q^2-6}{p^4q^2+6}+\frac{12}{p^4q^2+6}+\frac{p^6q^2+6p^2-2p^4q^2}{p^4q^2+6} \\
		& \qquad +\frac{p^4q^4+6q^2-2p^4q^2}{p^4q^2+6}+\frac{p^8q^2-p^6q^2-p^4q^2+6p^4-6p^2+6}{p^4q^2+6} \\
		& \qquad +\frac{p^6q^4-p^6q^2-p^4q^4+6p^2q^2-6p^2-6q^2+12}{p^4q^2+6} \\
		& \qquad + \frac{p^8q^4-p^6q^4-p^8q^2+p^6q^2+5p^4q^2-6p^2q^2-6p^4+6p^2+6}{p^4q^2+6}.
%		&=\frac{2p^8q^4+72}{p^4q^2+6}.
	\end{align*}
Hence, we get the required expression.
%	This completes the proof.
\end{proof}

\begin{theorem}\label{all_energy_comparison_cyclic_p^2q}
	If $G$ is a cyclic group of order $p^2q$, where $p$ and $q$ are two distinct primes, then $\mathcal{B}(G)$ is hypoenergetic but not hyperenergetic,  L-hyperenergetic, Q-hyperenergetic and CN-hyperenergetic.
\end{theorem}
\begin{proof}
	By Theorem \ref{structure_cyclic_p^2q} and Theorem \ref{all_energy_cyclic_p^2q}, we have $|V(\mathcal{B}(G))|=p^4q^2+6$ and $E(\mathcal{B}(G))=2+2\sqrt{p^2-1}+2\sqrt{q^2-1}+2p\sqrt{p^2-1}+2\sqrt{(p^2-1)(q^2-1)}+2p\sqrt{(p^2-1)(q^2-1)}$.
	Since $p^2>2\sqrt{p^2-1}, q^2>2\sqrt{q^2-1}$, $p^4-p^2+1>2\sqrt{p^4-p^2}=2p\sqrt{p^2-1}$, $p^2q^2-p^2-q^2+2 >2\sqrt{p^2q^2-p^2-q^2+1}=2\sqrt{(p^2-1)(q^2-1)}$ and $p^4q^2-p^2q^2-p^4+p^2+1>2\sqrt{p^4q^2-p^2q^2-p^4+p^2}=2p\sqrt{(p^2-1)(q^2-1)}$ we have 
	\begin{align*}
		p^2+q^2+p^4-p^2+1&+p^2q^2-p^2-q^2+2+p^4q^2-p^2q^2-p^4+p^2+1+2 \\
		&> 2+2\sqrt{p^2-1}+2\sqrt{q^2-1}+2p\sqrt{p^2-1}+2\sqrt{(p^2-1)(q^2-1)} \\
		& \qquad \qquad \qquad \qquad \qquad \qquad \qquad \qquad \qquad +2p\sqrt{(p^2-1)(q^2-1)}.
	\end{align*}
	Therefore,
	\begin{equation} \label{cyclic-p^2q-hypo}
		|V(\mathcal{B}(G))|> E(\mathcal{B}(G)).
	\end{equation}	 
	Thus, $\mathcal{B}(G)$ is hypoenergetic. 
	
	We have $E(K_{|V(\mathcal{B}(G))|}) = E(K_{p^4q^2+6})= 2(p^4q^2+6-1)=2p^4q^2+10>p^4q^2+6= |V(\mathcal{B}(G))|> E(\mathcal{B}(G))$ (using \eqref{cyclic-p^2q-hypo}). Therefore, $\mathcal{B}(G)$ is not hyperenergetic.

	Also, $LE(K_{p^4q^2+6})=LE^+(K_{p^4q^2+6})=2p^4q^2+10$. From Theorem \ref{all_energy_cyclic_p^2q}, we have $LE(\mathcal{B}(G))$ $=LE^+(\mathcal{B}(G))=\frac{2p^8q^4+72}{p^4q^2+6}$. Now, 
	\begin{align*}
		LE(K_{p^4q^2+6})-LE(\mathcal{B}(G))&= LE^+(K_{p^4q^2+6})-LE^+(\mathcal{B}(G))\\ &=2p^4q^2+10-\frac{2p^8q^4+72}{p^4q^2+6} 
		= \frac{22p^4q^2-12}{p^4q^2+6} > 0.
	\end{align*}
	Therefore, 
	$
	LE^+(K_{p^4q^2+6}) = LE(K_{p^4q^2+6}) > LE(\mathcal{B}(G))=LE^+(\mathcal{B}(G)).
	$
	Hence,  $\mathcal{B}(G)$ is neither L-hyperenergetic nor  Q-hyperenergetic.
	
	We have  $E_{CN}(K_{p^4q^2+6})=2(p^4q^2+6-1)(p^4q^2+6-2)=(2p^4q^2+10)(p^4q^2+4) > 2p^4q^2-12 = E_{CN}(\mathcal{B}(G))$ (using Theorem \ref{all_energy_cyclic_p^2q}). 
	Hence, $\mathcal{B}(G)$ is not CN-hyperenergetic. This completes the proof.
\end{proof}
\begin{theorem}
	If $G$ is a cyclic group of order $p^2q$, where $p$ and $q$ are two distinct primes, then $E(\mathcal{B}(G))<LE(\mathcal{B}(G))$.
\end{theorem}
\begin{proof}
	By Theorem \ref{all_energy_cyclic_p^2q}, we get
	$E(\mathcal{B}(G)) =2+2\sqrt{p^2-1}((1+p)(1+\sqrt{q^2-1}))+2\sqrt{q^2-1}$ and	$LE(\mathcal{B}(G)) =\frac{2p^8q^4+72}{p^4q^2+6}$. Also, from Theorem \ref{all_energy_comparison_cyclic_p^2q}, we have $|V(\mathcal{B}(G))|=p^4q^2+6> E(\mathcal{B}(G))$. Now, for all $p, q \geq 2$, we have
	\begin{align*}
		LE(\mathcal{B}(G))-|V(\mathcal{B}(G))|&=\frac{2p^8q^4+72}{p^4q^2+6}-p^4q^2+6 \\
		&= \frac{p^8q^4-12p^4q^2+36}{p^4q^2+6} 
		= \frac{p^4q^2(p^4q^2-12)+36}{p^4q^2+6} >0.
	\end{align*}
	Hence, $LE(\mathcal{B}(G))>|V(\mathcal{B}(G))|>E(\mathcal{B}(G))$.
\end{proof}

\begin{theorem}\label{all_spec_cyclic_p^2q^2}
If $G$ is a cyclic group of order $p^2q^2$, where $p$ and $q$ are two  primes such that $p < q$, then
    \begin{align*}
		\Spec(\mathcal{B}(G))=&\left\lbrace(0)^{p^4q^4-9}, (\pm 1)^1,  \left(\pm\sqrt{p^2-1}\right)^1,  \left(\pm\sqrt{q^2-1}\right)^1, \left(\pm p\sqrt{p^2-1}\right)^1,\right. \\
		& \left. \left(\pm q\sqrt{q^2-1}\right)^1,  
        \left(\pm\sqrt{(p^2-1)(q^2-1)}\right)^1,  \left(\pm p\sqrt{(p^2-1)(q^2-1)}\right)^1, \right. \\
        & \left.   \left(\pm q\sqrt{(p^2-1)(q^2-1)}\right)^1,  \left(\pm pq\sqrt{(p^2-1)(q^2-1)}\right)^1 \right\rbrace, 
	\end{align*}        
	\begin{align*}
	\L-Spec(\mathcal{B}(G))&=\left\lbrace(0)^{9}, (1)^{p^4q^4-9}, (2)^1, (p^2)^1, (q^2)^1, (p^4-p^2+1)^1, (q^4-q^2+1)^1, \right. \\
    & \qquad  \left. (p^2q^2-p^2-q^2+2)^1, (p^4q^2-p^2q^2-p^4+p^2+1)^1, \right. \\
    & \qquad \left. (p^2q^4-p^2q^2-q^4+q^2+1)^1, (p^4q^4-p^2q^4-p^4q^2+p^2q^2+1)^1\right\rbrace \\
    & = \Q-Spec(\mathcal{B}(G))
	\end{align*}
    \begin{align*}
\text{and }    \CN-Spec(\mathcal{B}(G))&=\left \lbrace(0)^{10}, (-1)^{p^4q^4-9},(p^2-2)^1, (q^2-2)^1, (p^4-p^2-1)^1,  \right. \\
    & \qquad \left. (q^4-q^2-1)^1, (p^2q^2-p^2-q^2)^1, (p^4q^2-p^2q^2-p^4+p^2-1)^1, \right. \\
    & \qquad \left. (p^2q^4-p^2q^2-q^4+q^2-1)^1, (p^4q^4-p^2q^4-p^4q^2+p^2q^2-1)^1 \right\rbrace.
\end{align*}
\end{theorem}
\begin{proof}
From Theorem \ref{structure_cyclic_p^2q^2}, we have \, $\, \mathcal{B}(G)\, =\, K_2 \sqcup K_{1, p^2-1} \sqcup K_{1, p^4-p^2} \sqcup K_{1, q^2-1} \sqcup K_{1, q^4-q^2} \sqcup K_{1, p^2q^2-p^2-q^2+1} \sqcup K_{1, p^2q^4-p^2q^2-q^4+q^2} \sqcup K_{1, p^4q^2-p^2q^2-p^4+p^2} \sqcup K_{1, p^4q^4-p^2q^4-p^4q^2+p^2q^2}$. Now, by Lemma \ref{all_spec_of_star}, we get 
\begin{center}
    $\Spec(K_2)=\{(\pm 1)^1 \, \Spec(K_{1, p^2-1})=\left\{(0)^{p^2-2}, (\pm \sqrt{p^2-1})^1\right\},$ $\Spec(K_{1, p^4-p^2})=\left\{(0)^{p^4-p^2-1}, (\pm p\sqrt{p^2-1})^1\right\},$
    $\Spec(K_{1, q^2-1})=\left\{(0)^{q^2-2}, (\pm \sqrt{q^2-1})^1\right\},$ $\Spec(K_{1, q^4-q^2})=\left\{(0)^{q^4-q^2-1}, (\pm q\sqrt{q^2-1})^1\right\},$
    $\Spec(K_{1, p^2q^2-p^2-q^2+1})=\left\{(0)^{p^2q^2-p^2-q^2}, (\pm\sqrt{(p^2-1)(q^2-1)})^1\right\}$, 
\end{center}
\begin{align*}
    \Spec(K_{1, p^2q^4-p^2q^2-q^4+q^2})&=\left\{(0)^{p^2q^4-p^2q^2-q^4+q^2-1}, (\pm q\sqrt{(p^2-1)(q^2-1)})^1\right\},
\end{align*}   
\begin{align*}
    \Spec(K_{1, p^4q^2-p^2q^2-p^4+p^2})&=\left\{(0)^{p^4q^2-p^2q^2-p^4+p^2-1}, (\pm p\sqrt{(p^2-1)(q^2-1)})^1\right\} 
\end{align*}    
and $\Spec(K_{1, p^4q^4-p^2q^4-p^4q^2+p^2q^2})=\left\{(0)^{p^4q^4-p^2q^4-p^4q^2+p^2q^2-1}, (\pm pq\sqrt{(p^2-1)(q^2-1)})^1\right\}$.
  
Since $\Spec(\mathcal{B}(G))$ is the union of the above multi-sets,  we get the desired expression for $\Spec(\mathcal{B}(G))$. Also, by Lemma \ref{all_spec_of_star}, we get
\begin{center}
    $\L-Spec(K_2)=\{(0)^1, (2)^1\},$ $ \L-Spec(K_{1, p^2-1})=\left\{ (0)^1, (1)^{p^2-2}, (p^2)^1\right\},$ $ \L-Spec(K_{1, p^4-p^2})=\left\{ (0)^1, (1)^{p^4-p^2-1}, (p^4-p^2+1)^1\right\},$ $\L-Spec(K_{1, q^2-1})=\left\{ (0)^1, (1)^{q^2-2}, (q^2)^1\right\}, $ $ \L-Spec(K_{1, q^4-q^2})=\left\{ (0)^1, (1)^{q^4-q^2-1}, (q^4-q^2+1)^1\right\},$
\end{center} 
\vspace{-0.3cm}
\begin{align*}
    \L-Spec(K_{1, p^2q^2-p^2-q^2+1})=\left\{(0)^1, (1)^{p^2q^2-p^2-q^2}, (p^2q^2-p^2-q^2+2)^1\right\},
\end{align*}  
\begin{align*}
    \L-Spec(K_{1, p^2q^4-p^2q^2-q^4+q^2})=\left\{(0)^1, (1)^{p^2q^4-p^2q^2-q^4+q^2-1}, (p^2q^4-p^2q^2-q^4+q^2+1)^1\right\}, 
\end{align*}
\begin{align*}
    \L-Spec(K_{1, p^4q^2-p^2q^2-p^4+p^2})=\left\{(0)^1, (1)^{p^4q^2-p^2q^2-p^4+p^2-1}, (p^4q^2-p^2q^2-p^4+p^2+1)^1\right\} 
\end{align*}
and $\L-Spec(K_{1, p^4q^4-p^2q^4-p^4q^2+p^2q^2})=$
\begin{align*}
   \left\{(0)^1, (1)^{p^4q^4-p^2q^4-p^4q^2+p^2q^2-1}, (p^4q^4-p^2q^4-p^4q^2+p^2q^2+1)^1\right\}.
\end{align*}
Since $\L-Spec(\mathcal{B}(G))$ is the union of the above multi-sets,  we get the desired expression for $\L-Spec(\mathcal{B}(G))$ and hence the expression for $\Q-Spec(\mathcal{B}(G))$. Further, we have
\begin{center}
    $\CN-Spec(K_2)=\{(0)^2\}, \, \CN-Spec(K_{1, p^2-1})=\left\{(0)^1, (-1)^{p^2-2}, (p^2-2)^1\right\},$ $\CN-Spec(K_{1, p^4-p^2})=\left\{(0)^1, (-1)^{p^4-p^2-1}, (p^4-p^2-1)^1\right\},$ 
    $\CN-Spec(K_{1, q^2-1})=\left\{(0)^1, (-1)^{q^2-2}, (q^2-2)^1\right\},$ $\CN-Spec(K_{1, q^4-q^2})=\left\{(0)^1, (-1)^{q^4-q^2-1}, (q^4-q^2-1)^1\right\},$  $\CN-Spec(K_{1, p^2q^2-p^2-q^2+1})=\left\{(0)^1, (-1)^{p^2q^2-p^2-q^2}, (p^2q^2-p^2-q^2)^1\right\}$, 
\end{center}
\begin{align*}
    \CN-Spec(K_{1, p^2q^4-p^2q^2-q^4+q^2})&\\
   =& \left\{(0)^1, (-1)^{p^2q^4-p^2q^2-q^4+q^2-1},  (p^2q^4-p^2q^2-q^4+q^2-1)^1\right\},
\end{align*} 
\begin{align*}
    \CN-Spec(K_{1, p^4q^2-p^2q^2-p^4+p^2})&\\
    =&\left\{(0)^1, (-1)^{p^4q^2-p^2q^2-p^4+p^2-1},  (p^4q^2-p^2q^2-p^4+p^2-1)^1\right\} 
\end{align*} 
\begin{align*}
 \text{ and }   \CN-Spec&(K_{1, p^4q^4-p^2q^4-p^4q^2+p^2q^2})\\
 &\qquad =\left\{(0)^1, (-1)^{p^4q^4-p^2q^4-p^4q^2+p^2q^2-1},  (p^4q^4-p^2q^4-p^4q^2+p^2q^2-1)^1\right\}.
\end{align*} 
Since $\CN-Spec(\mathcal{B}(G))$ is the union of the above multi-sets,  we get the desired expression of $\CN-Spec(\mathcal{B}(G))$.    
\end{proof}
\begin{theorem}\label{all_energy_cyclic_p^2q^2}
If $G$ is a cyclic group of order $p^2q^2$, where $p$ and $q$ are two  primes such that $p < q$, then
    $E(\mathcal{B}(G))= 2+2(1+p)\sqrt{p^2-1}+2(1+q)\sqrt{q^2-1}+2(1+p+q+pq)\sqrt{(p^2-1)(q^2-1)}$,
	$LE(\mathcal{B}(G))= LE^+(\mathcal{B}(G))=\frac{2p^8q^8+162}{p^4q^4+9}$ 
	\quad  and \,
	$E_{CN}(\mathcal{B}(G))=2p^4q^4-18$.
\end{theorem}
\begin{proof}
    From Theorem \ref{all_spec_cyclic_p^2q^2} and definition of $E(\mathcal{B}(G))$ and $E_{CN}(\mathcal{B}(G))$, we have
    \begin{align*}
        E(\mathcal{B}&(G))=2+2\sqrt{p^2-1}+2\sqrt{q^2-1}+2p\sqrt{p^2-1}+2q\sqrt{q^2-1}+2\sqrt{(p^2-1)(q^2-1)} \\
        & \qquad \qquad +2p\sqrt{(p^2-1)(q^2-1)}+2q\sqrt{(p^2-1)(q^2-1)}+2pq\sqrt{(p^2-1)(q^2-1)} 
        %&= 2+2(1+p)\sqrt{p^2-1}+2(1+q)\sqrt{q^2-1}+2(1+p+q+pq)\sqrt{(p^2-1)(q^2-1)} 
    \end{align*}
    \begin{align*}
     \text{ and } \, E_{CN}(\mathcal{B}(G))&=(\underbrace{1+1+\cdots+1}
        _{(p^4q^4-9)\text{-times}})+p^2-2+q^2-2+p^4-p^2-1+q^4-q^2-1 \\
        & \qquad \qquad \qquad +p^2q^2-p^2-q^2+p^4q^2-p^2q^2-p^4+p^2-1+p^2q^4 \\
        & \qquad \qquad \qquad -p^2q^2-q^4+q^2-1+p^4q^4-p^2q^4-p^4q^2+p^2q^2-1.
        %&= 2p^4q^4-18.
    \end{align*}
Thus we get the required expressions for $E(\mathcal{B}(G))$ and $E_{CN}(\mathcal{B}(G))$.

Further, from Theorem \ref{all_spec_cyclic_p^2q^2} and Theorem \ref{structure_cyclic_p^2q^2}, we have 

\noindent $\L-Spec(\mathcal{B}(G))= \Q-Spec(\mathcal{B}(G))=\lbrace(0)^{9}, (1)^{p^4q^4-9}, (2)^1, (p^2)^1, (q^2)^1, (p^4-p^2+1)^1, (q^4-q^2+1)^1, (p^2q^2-p^2-q^2+2)^1, (p^4q^2 -p^2q^2-p^4+p^2+1)^1, (p^2q^4-p^2q^2-q^4+q^2+1)^1, (p^4q^4-p^2q^4-p^4q^2+p^2q^2+1)^1 \rbrace $ 

\noindent and $\frac{2m}{n}=\frac{2p^4q^4}{p^4q^4+9}$. Now, $|0-\frac{2p^4q^4}{p^4q^4+9}|=\frac{2p^4q^4}{p^4q^4+9}$, $|1-\frac{2p^4q^4}{p^4q^4+9}|=\frac{p^4q^4-9}{p^4q^4+9}$, $|2-\frac{2p^4q^4}{p^4q^4+9}|=\frac{18}{p^4q^4+9}$, $|p^2-\frac{2p^4q^4}{p^4q^4+9}|=\frac{p^6q^4+9p^2-2p^4q^4}{p^4q^4+9}$, $|q^2-\frac{2p^4q^4}{p^4q^4+9}|=\frac{p^4q^6+9q^2-2p^4q^4}{p^4q^4+9}$, $|(p^4-p^2+1)-\frac{2p^4q^4}{p^4q^4+9}|=\frac{p^8q^4-p^6q^4-p^4q^4+9p^4-9p^2+9}{p^4q^4+9}$, $|(q^4-q^2+1)-\frac{2p^4q^4}{p^4q^4+9}|\!=\!\frac{p^4q^8\!-p^4q^6\!-p^4q^4+9q^4-9q^2+9}{p^4q^4+9}$, $|(p^2q^2-p^2-q^2+2)-\frac{2p^4q^4}{p^4q^4+9}|\!=\!\frac{p^6q^6-p^6q^4-p^4q^6+9p^2q^2-9p^2-9q^2+18}{p^4q^4+9}$, $|(p^2q^4-p^2q^2-q^4+q^2+1)-\frac{2p^4q^4}{p^4q^4+9}|=\frac{p^6q^8-p^6q^6-p^4q^8+p^4q^6-p^4q^4+9p^2q^4-9p^2q^2-9q^4+9q^2+9}{p^4q^4+9}$, $|(p^4q^2-p^2q^2-p^4+p^2+1)-\frac{2p^4q^4}{p^4q^4+9}|=\frac{p^8q^6-p^6q^6-p^8q^4+p^6q^4-p^4q^4+9p^4q^2-9p^2q^2-9p^4+9p^2+9}{p^4q^4+9}$ and $|(p^4q^4-p^2q^4-p^4q^2+p^2q^2+1)-\frac{2p^4q^4}{p^4q^4+9}|=\frac{p^8q^8-p^6q^8-p^8q^6+p^6q^6+8p^4q^4-9p^2q^4-9p^4q^2+9p^2q^2+9}{p^4q^4+9}$. Therefore, by definition of (signless) Laplacian energy, we have
\begin{align*}
    LE(\mathcal{B}(G))&= LE^+(\mathcal{B}(G))\\
    & = 9 \times \frac{2p^4q^4}{p^4q^4+9}+(p^4q^4-9)\frac{p^4q^4-9}{p^4q^4+9}+\frac{18}{p^4q^4+9}+\frac{p^6q^4+9p^2-2p^4q^4}{p^4q^4+9} \\
    & \qquad +\frac{p^4q^6+9q^2-2p^4q^4}{p^4q^4+9}+\frac{p^8q^4-p^6q^4-p^4q^4+9p^4-9p^2+9}{p^4q^4+9} \\
    & \qquad +\frac{p^4q^8-p^4q^6-p^4q^4+9q^4-9q^2+9}{p^4q^4+9} \\
    & \qquad + \frac{p^6q^6-p^6q^4-p^4q^6+9p^2q^2-9p^2-9q^2+18}{p^4q^4+9}\\
    & \qquad + \frac{p^6q^8-p^6q^6-p^4q^8+p^4q^6-p^4q^4+9p^2q^4-9p^2q^2-9q^4+9q^2+9}{p^4q^4+9} \\
    & \qquad +\frac{p^8q^6-p^6q^6-p^8q^4+p^6q^4-p^4q^4+9p^4q^2-9p^2q^2-9p^4+9p^2+9}{p^4q^4+9} \\
    & \qquad +\frac{p^8q^8-p^6q^8-p^8q^6+p^6q^6+8p^4q^4-9p^2q^4-9p^4q^2+9p^2q^2+9}{p^4q^4+9}.
%    &=\frac{2p^8q^8+162}{p^4q^4+9}.
\end{align*}
Hence, we get the required expression.
%This completes the proof.
\end{proof}

\begin{theorem}\label{all_energy_comparison_cyclic_p^2q^2}
	If $G$ is a cyclic group of order $p^2q^2$, where $p$ and $q$ are any distinct primes, then $\mathcal{B}(G)$ is hypoenergetic but not hyperenergetic,  L-hyperenergetic, Q-hyperenergetic and CN-hyperenergetic.
\end{theorem}
\begin{proof}
	By Theorem \ref{structure_cyclic_p^2q^2} and Theorem \ref{all_energy_cyclic_p^2q^2}, we have $|V(\mathcal{B}(G))|=p^4q^4+9$ and $E(\mathcal{B}(G))=2 \, + \, 2\sqrt{p^2-1}\, + \, 2\sqrt{q^2-1}\, + \, 2p\sqrt{p^2-1} \, + \, 2q\sqrt{q^2-1} \, + \, 2\sqrt{(p^2-1)(q^2-1)} \, + \, \\
    2p\sqrt{(p^2-1)(q^2-1)} +  2q\sqrt{(p^2-1)(q^2-1)} + 2pq\sqrt{(p^2-1)(q^2-1)}$. Since $p^2>2\sqrt{p^2-1}$, $q^2>2\sqrt{q^2-1}$, $p^4-p^2+1>2\sqrt{p^4-p^2}  =2p\sqrt{p^2-1}$, $q^4-q^2+1>2\sqrt{q^4-q^2} =2q\sqrt{q^2-1}$, $p^2q^2-p^2-q^2+2 >2\sqrt{p^2q^2-p^2-q^2+1}  = 2\sqrt{(p^2-1)(q^2-1)}$, $p^4q^2-p^2q^2-p^4+p^2+1>2\sqrt{p^4q^2-p^2q^2-p^4+p^2}=2p\sqrt{(p^2-1)(q^2-1)}$, $p^2q^4-p^2q^2-q^4+q^2+1> 2\sqrt{p^2q^4-p^2q^2-q^4+q^2}=2q\sqrt{(p^2-1)(q^2-1)}$ and $p^4q^4-p^2q^4-p^4q^2+p^2q^2+1> 2\sqrt{p^4q^4-p^2q^4-p^4q^2+p^2q^2}=2pq\sqrt{(p^2-1)(q^2-1)}$ we have 
	\begin{align*}
	    p^2+q^2+&p^4-p^2+1+q^4-q^2+1+p^2q^2-p^2-q^2+2+p^4q^2-p^2q^2-p^4 \\
        & +p^2+1+p^2q^4-p^2q^2-q^4+q^2+1+p^4q^4-p^2q^4-p^4q^2+p^2q^2+1+2 \\
        &> 2 + 2\sqrt{p^2-1} +  2\sqrt{q^2-1} +  2p\sqrt{p^2-1} +  2q\sqrt{q^2-1} + 2\sqrt{(p^2-1)(q^2-1)}  \\
        & \qquad + 2p\sqrt{(p^2-1)(q^2-1)} +  2q\sqrt{(p^2-1)(q^2-1)} + 2pq\sqrt{(p^2-1)(q^2-1)}.
	\end{align*}
	Therefore,
	\begin{equation} \label{cyclic-p^2q^2-hypo}
		|V(\mathcal{B}(G))|> E(\mathcal{B}(G)).
	\end{equation}	 
	Thus, $\mathcal{B}(G)$ is hypoenergetic. 
	
	We have $E(K_{|V(\mathcal{B}(G))|}) = E(K_{p^4q^4+9})= 2(p^4q^4+9-1)=2p^4q^4+16>p^4q^4+9= |V(\mathcal{B}(G))|> E(\mathcal{B}(G))$ (using \eqref{cyclic-p^2q^2-hypo}). Therefore, $\mathcal{B}(G)$ is not hyperenergetic.

	Also, $LE(K_{p^4q^4+9})=LE^+(K_{p^4q^4+9})=2p^4q^4+16$. From Theorem \ref{all_energy_cyclic_p^2q}, we have $LE(\mathcal{B}(G))$ $=LE^+(\mathcal{B}(G))=\frac{2p^8q^8+162}{p^4q^4+9}$. Now, 
	\begin{align*}
		LE(K_{p^4q^4+9})-LE(\mathcal{B}(G))&= LE^+(K_{p^4q^4+9})-LE^+(\mathcal{B}(G))\\ &=2p^4q^4+16-\frac{2p^8q^8+162}{p^4q^4+9} 
		= \frac{34p^4q^4-18}{p^4q^4+9} > 0.
	\end{align*}
Therefore, 
	$
	LE^+(K_{p^4q^4+9}) = LE(K_{p^4q^4+9}) > LE(\mathcal{B}(G))=LE^+(\mathcal{B}(G)).
	$
	Hence,  $\mathcal{B}(G)$ is neither L-hyperenergetic nor  Q-hyperenergetic.
	
	We have  $E_{CN}(K_{p^4q^4+9})=2(p^4q^4+9-1)(p^4q^4+9-2)=(2p^4q^4+16)(p^4q^4+7) > 2p^4q^4-18 = E_{CN}(\mathcal{B}(G))$ (using Theorem \ref{all_energy_cyclic_p^2q^2}). 
	Hence, $\mathcal{B}(G)$ is not CN-hyperenergetic. This completes the proof.
\end{proof}

 \begin{theorem}
	If $G$ is a cyclic group of order $p^2q^2$, where $p$ and $q$ are two primes such that $p < q$, then $E(\mathcal{B}(G))<LE(\mathcal{B}(G))$.
\end{theorem}
\begin{proof}
	By Theorem \ref{all_energy_cyclic_p^2q^2}, we get
	$E(\mathcal{B}(G)) =2+2(1+p)\sqrt{p^2-1}+2(1+q)\sqrt{q^2-1}+2(1+p+q+pq)\sqrt{(p^2-1)(q^2-1)}$ and $LE(\mathcal{B}(G)) =\frac{2p^8q^8+162}{p^4q^4+9}$. Also, from Theorem \ref{all_energy_comparison_cyclic_p^2q^2}, we have $|V(\mathcal{B}(G))|=p^4q^4+9> E(\mathcal{B}(G))$. Now, for all $p, q \geq 2$, we have
 \begin{align*}
     LE(\mathcal{B}(G))-|V(\mathcal{B}(G))|=\frac{2p^8q^8+162}{p^4q^4+9}-p^4q^4+9
     = \frac{p^8q^8-18p^4q^4+81}{p^4q^4+9} 
     = \frac{p^4q^4(p^4q^4-18)+81}{p^4q^4+9} >0.
 \end{align*}
 Hence, $LE(\mathcal{B}(G))= LE^+(\mathcal{B}(G))>|V(\mathcal{B}(G))|>E(\mathcal{B}(G))$.
 \end{proof}
 We conclude this section noting that if $G$ is a cyclic group of order $p^n, pq, p^2q$ and $p^2q^2$ for any two distinct primes $p$ and $q$ and $n \geq 1$, then
 \begin{enumerate}
     \item $\mathcal{B}(G)$ is not integral but L-integral, Q-integral and CN-integral.
     \item $\mathcal{B}(G)$ is hypoenergetic but neither hyperenergetic, L-hyperenergetic, Q-hyper-energetic nor CN-hyperenergetic.
     \item $\mathcal{B}(G)$ satisfies E-LE conjecture.
 \end{enumerate}
 It may be interesting to conclude the same for any cyclic group.

%{\bf Acknowledgement.} 

%\section*{Declarations}
%\begin{itemize}
%	\item Funding: No funding was received by the authors.
%	\item Conflict of interest: The authors declare that they have no conflict of interest.
%	\item Availability of data and materials: No data was used in the preparation of this manuscript.
%%	\item Authors' contributions: all the authors contributed equally to this work.
%\end{itemize}

\end{document}